\documentclass[12pt, a4paper, leqno, twoside]{article}
\usepackage{amsmath, amsthm, amsfonts, amssymb, graphicx, color, cancel, ulem}
\usepackage{esint}
\numberwithin{equation}{section}
     \newtheorem{thm}{Theorem}[section]
     \newtheorem{cor}[thm]{Corollary}
     
     \newtheorem{lem}[thm]{Lemma}
     \newtheorem{defn}[thm]{Definition}
      \newtheorem{rem}[thm]{Remark}

\newcommand{\R}{\mathbb{R}}

\newcommand{\cY}{\mathcal{Y}}

\newcommand{\supp}{\operatorname{supp}}

\newcommand{\loc}{\mathrm{loc}}

\newcommand{\weak}{\mathrm{weak}}

\newcommand{\ls}{\lesssim}
\newcommand{\gs}{\gtrsim}

\newcommand{\trho}{\tilde\rho}

\newcommand{\vp}{\varphi}

\newcommand{\tvp}{\tilde\varphi}

\newcommand{\cPhi}{\widetilde\Phi}

\newcommand{\cGdec}{\mathcal{G}^{\rm dec}}

\newcommand{\LP}{L^{\Phi}}
\newcommand{\LPs}{L^{\Psi}}

\newcommand{\LcP}{L^{\widetilde\Phi}}

\newcommand{\wL}{\mathrm{w}\hskip-0.6pt{L}}
\newcommand{\wLP}{\mathrm{w}\hskip-0.6pt{L}^{\Phi}}
\newcommand{\wLPs}{\mathrm{w}\hskip-0.6pt{L}^{\Psi}}
\newcommand{\LPp}{L^{(\Phi,\vp)}}

\newcommand{\wLPp}{\mathrm{w}\hskip-0.6pt{L}^{(\Phi,\vp)}}

\newcommand{\LPsp}{L^{(\Psi,\vp)}}
\newcommand{\wLPsp}{\mathrm{w}\hskip-0.6pt{L}^{(\Psi,\vp)}}
\newcommand{\LPsps}{L^{(\Psi,\psi)}}
\newcommand{\wLPsps}{\mathrm{w}\hskip-0.6pt{L}^{(\Psi,\psi)}}

\newcommand{\Ir}{I_{\rho}}
\newcommand{\Ia}{I_{\alpha}}

\newcommand{\Ma}{M_{\alpha}}
\newcommand{\Mr}{M_{\rho}}

\newcommand{\biP}{{\it{\overline\varPhi}}}
\newcommand{\iPy}{{\it{\Phi_Y}}}

\newcommand{\biPy}{{\it{\overline\Phi_Y}}}

\newcommand{\dtwo}{\Delta_2}
\newcommand{\ntwo}{\nabla_2}
\newcommand{\bdtwo}{\overline\Delta_2}
\newcommand{\bntwo}{\overline\nabla_2}

\newcommand{\dlim}{\displaystyle\lim}

\newcommand{\msckw}{%
\footnotetext{\hspace{-0.35cm} 2010 {\it Mathematics Subject Classification}. 
42B35, 46E30, 42B20, 42B25.
\endgraf{\it Key words and phrases.} 
Fractional integral, fractional maximal operator, Orlicz-Morrey space, weak Orlicz-Morrey space, modular inequality. 
}
}

\begin{document}

\title{%
Characterization of the boundedness of 
generalized fractional integral and maximal operators
on Orlicz-Morrey and weak Orlicz-Morrey spaces
\msckw
}
\author{%
Ryota Kawasumi, Eiichi Nakai and Minglei Shi}
\date{}

\maketitle

\begin{abstract}
We give necessary and sufficient conditions 
for the boundedness of 
generalized fractional integral and maximal operators
on Orlicz-Morrey and weak Orlicz-Morrey spaces.
To do this we prove the weak-weak type modular inequality of the Hardy-Littlewood maximal operator
with respect to the Young function.
Orlicz-Morrey spaces contain $L^p$ spaces ($1\le p\le\infty$), Orlicz spaces 
and generalized Morrey spaces as special cases.
Hence we get necessary and sufficient conditions 
on these function spaces
as corollaries.
\end{abstract}
\maketitle                   







\section{Introduction}\label{sec:intro}

In this paper 
we consider the generalized fractional integral operator $\Ir$ 
and the generalized fractional maximal operator $\Mr$.
We give necessary and sufficient conditions 
for the boundedness of $\Ir$ and $\Mr$
on Orlicz-Morrey spaces $\LPp(\R^n)$ and weak Orlicz-Morrey spaces $\wLPp(\R^n)$.
More precisely, we investigate the following boundedness of the operator $T=\Ir$ or $\Mr$:
\begin{align*}
 \|Tf\|_{\wLPsp}&\le C\|f\|_{\LPp}, \\
 \|Tf\|_{\LPsp}&\le C\|f\|_{\LPp}, \\
 \|Tf\|_{\wLPsp}&\le C\|f\|_{\wLPp}.
\end{align*}
We treat a wide class of Young functions as $\Phi:[0,\infty]\to[0,\infty]$.
Orlicz-Morrey spaces contain $L^p$ spaces ($1\le p\le\infty$), Orlicz spaces 
and generalized Morrey spaces as special cases.
Hence we get necessary and sufficient conditions 
for the boundedness of $\Ir$ and $\Mr$ on these function spaces
as corollaries.

Now, we recall the definitions of the Orlicz-Morrey and weak Orlicz-Morrey spaces.
For a measurable set $G \subset \R^n$, 
we denote by $|G|$ its Lebesgue measure.
We denote by $B(a,r)$ 
the open ball centered at $a\in\R^n$ and of radius $r$.
For a function $f\in L^1_{\loc}(\R^n)$ and a ball $B$, 
let 
\begin{equation*}\label{mean}
  f_{B}=\fint_{B} f=\fint_{B} f(y)\,dy=\frac1{|B|}\int_{B} f(y)\,dy.
\end{equation*}
For a measurable set $G\subset\R^n$, 
a measurable function $f$ and $t\ge0$, 
let 
\begin{equation*} 
 m(G,f,t)=|\{x\in G: |f(x)|>t\}|.
\end{equation*}
In the case $G=\R^n$, we briefly denote it by $m(f,t)$.

\begin{defn}\label{defn:OM}
For a Young function $\Phi:[0,\infty]\to[0,\infty]$,
a function $\vp:(0,\infty)\to(0,\infty)$
and a ball $B=B(a,r)$,
let 
\begin{align*}
     \|f\|_{\Phi,\vp,B}
     &= 
     \inf\left\{ \lambda>0: 
     \frac{1}{\vp(r)}
       \fint_B \!\Phi\!\left(\frac{|f(x)|}{\lambda}\right)\! dx \le 1
         \right\},
\\
     \|f\|_{\Phi,\vp,B,\weak}
     &= 
     \inf\left\{ \lambda>0: 
     \sup_{t\in(0,\infty)} 
      \frac{\Phi(t)\, m(B,f/\lambda,t)}{|B|\vp(r)}
       \le 1
         \right\}.
\end{align*}
Let $\LPp(\R^n)$ and $\wLPp(\R^n)$ 
be the set of all measurable functions $f$ on $\R^n$
such that the following functionals are finite, respectively: 
\begin{align*}
     \|f\|_{\LPp} 
        &= \sup_{B}  \|f\|_{\Phi,\vp,B}, 
\\
     \|f\|_{\wLPp} 
        &= \sup_{B}  \|f\|_{\Phi,\vp,B,\weak},
\end{align*}
where the suprema are taken over all balls $B$ in $\R^n$.
(For the definition of the Young function, see Definition~\ref{defn:Young} below.)
\end{defn}

Then $\|\cdot\|_{\LPp(\R^n)}$ is a norm 
and thereby $\LPp(\R^n)$ is a Banach space.
On the other hand $\|\cdot\|_{\wLPp(\R^n)}$ is a quasi norm with
the inequality
\begin{equation*}
\|f+g\|_{\wLPp}\le2\left(\|f\|_{\wLPp}+\|g\|_{\wLPp}\right),
\end{equation*}
and thereby $\wLPp(\R^n)$ is a quasi Banach space.

If $\vp(r)=1/r^n$, then $\LPp(\R^n)=\LP(\R^n)$ 
and $\wLPp(\R^n)=\wLP(\R^n)$,
which are the usual Orlicz and weak Orlicz spaces, respectively.
In this paper we treat a wide class of Young functions such as
\begin{equation}\label{Young}
 \Phi_1(t)=
\begin{cases}
 t, & 0\le t\le1, \\
 \infty, & t>1,
\end{cases}
\qquad
\Phi_2(t)=\max(0,t^2-4), \quad t \ge 0.
\end{equation}
We also treat generalized Young functions such as
\begin{equation}\label{gYoung}
 \Phi_3(t)=
\begin{cases}
 e^{1-1/t^p}, & 0\le t\le1, \\
 e^{t^p-1}, & t>1,
\end{cases}
\quad 0<p<\infty,
\end{equation}
which is not convex near $t=1$.

If $\Phi(t)=t^p$ $(1\le p<\infty)$, 
then we denote 
$\LPp(\R^n)$ and $\wLPp(\R^n)$
by 
$L^{(p,\vp)}(\R^n)$ and $\wL^{(p,\vp)}(\R^n)$,
which are the generalized Morrey and weak Morrey spaces,
respectively.
If $\vp_{\lambda}(r)=r^{\lambda}$, $-n\le\lambda<0$,
then $L^{(p,\vp_{\lambda})}(\R^n)$ is the classical Morrey space.
In particular, if $\lambda=-n$, then it is the Lebesgue space $L^p(\R^n)$.

Orlicz spaces were introduced by \cite{Orlicz1932,Orlicz1936}.
For the theory of Orlicz spaces,
see 
\cite{Kita2009,Kokilashvili-Krbec1991,Krasnoselsky-Rutitsky1961,Maligranda1989,Rao-Ren1991}
for example.
Weak Orlicz spaces were studied in
\cite{Jiao-Peng-Liu2008,Liu-Wang2013,Nakai2002Lund}, etc.
Morrey spaces were introduced by \cite{Morrey1938}.
For their generalization, 
see \cite{Mizuhara1990,Nakai1994MathNachr,Peetre1966,Peetre1969}, etc.
Weak Morrey spaces were studied in
\cite{Gunawan-Hakim-Limanta-Masta2017,Gunawan-Hakim-Nakai-Sawano2018,Sawano-ElShabrawy2018,Sihwaningrum-Sawano2013,Tumalun-Gunawan2019}, etc.
The Orlicz-Morrey space $\LPp(\R^n)$ was first studied in \cite{Nakai2004KIT}.
The spaces $\LPp(\R^n)$ and $\wLPp(\R^n)$ 
were investigated 
in \cite{Ho2013,Nakai2008Studia,Nakai2008KIT,Sawano2019}, etc.
For other kinds of Orlicz-Morrey spaces, see 
\cite{Deringoz-Guliyev-Nakai-Sawano-Shi2019Posi,Deringoz-Guliyev-Samko2014,Gala-Sawano-Tanaka2015,Guliyev-Hasanov-Sawano-Noi2016,Sawano-Sugano-Tanaka2012}, etc.
For the study related to weak Orlicz and weak Morrey spaces,
see \cite{Ferreira2016,Gogatishvili-Mustafayev-Agcayazi2018,Hakim-Sawano2016,Hatano2019preprint,Ho2017,Ho2019,Li2012,Liang-Yang-Jiang2016,Xie-Yang2019}, etc.

Next we recall 
the generalized fractional integral operator $\Ir$.
For a function $\rho:(0,\infty)\to(0,\infty)$,
the operator $\Ir$ is defined by
\begin{equation}\label{Ir}
 \Ir f(x)=\int_{\R^n}\frac{\rho(|x-y|)}{|x-y|^n}f(y)\,dy,
 \quad x\in\R^n,
\end{equation}
where we always assume that
\begin{equation}\label{int rho}
 \int_0^1\frac{\rho(t)}{t}\,dt<\infty.
\end{equation}
If $\rho(r)=r^{\alpha}$, $0<\alpha<n$, 
then $\Ir$ is the usual fractional integral operator $\Ia$.
The condition \eqref{int rho} is needed for the integral in \eqref{Ir} 
to converge for bounded functions $f$ with compact support.
In this paper we also assume that 
there exist positive constants $C$, $K_1$ and $K_2$ with $K_1<K_2$ such that, for all $r>0$,
\begin{equation}\label{sup rho}
 \sup_{r\le t\le 2r}\rho(t)
 \le
 C\int_{K_1r}^{K_2r}\frac{\rho(t)}{t}\,dt.
\end{equation}
The condition \eqref{sup rho} was considered in \cite{Perez1994}.
If $\rho$ satisfies the doubling condition \eqref{doubling} below,
then $\rho$ satisfies \eqref{sup rho}.
Let $\rho(r)=\min(r^{\alpha},e^{-r/2})$ with $0<\alpha<n$, 
which controls the Bessel potential (see \cite{Stein1970}).
Then $\rho$ also satisfies \eqref{sup rho}.
The operator $\Ir$ was introduced in \cite{Nakai2001Taiwan} 
to extend the Hardy-Littlewood-Sobolev theorem to Orlicz spaces
whose partial results were announced in 
\cite{Nakai2000ISAAC}.
For example, 
the generalized fractional integral $I_{\rho}$ 
is bounded from $\exp L^p(\R^n)$ to $\exp L^q(\R^n)$,
where 
\begin{equation}\label{rho log}
     \rho(r)=
     \begin{cases}
          1/(\log(1/r))^{\alpha+1} & \text{for small}\; r, \\
          (\log r)^{\alpha-1}     & \text{for large}\; r,
     \end{cases}
     \quad \alpha>0,
\end{equation}
$p,q\in(0,\infty)$, $-1/p+\alpha=-1/q$ 
and $\exp L^p(\R^n)$ is the Orlicz space $L^{\Phi}(\R^n)$ with 
\begin{equation}\label{Phi exp}
     \Phi(t)=
     \begin{cases}
          1/\exp(1/t^p) & \text{for small}\; t, \\
          \exp(t^p)     & \text{for large}\; t.
     \end{cases}
\end{equation}
See also 
\cite{Nakai2001SCMJ,Nakai2002Lund,Nakai2004KIT,Nakai2008Studia,Nakai-Sumitomo2001SCMJ}.

We also consider the generalized fractional maximal operator $\Mr$
and compare its boundedness with $\Ir$.
For a function $\rho:(0,\infty)\to(0,\infty)$, 
the operator $\Mr$ is defined by
\begin{equation}\label{Mr}
 \Mr f(x)=\sup_{B(a,r)\ni x}\rho(r)\fint_{B(a,r)}|f(y)|\,dy,
 \quad x\in\R^n,
\end{equation}
where the supremum is taken over all balls $B(a,r)$ containing $x$.
We need neither the condition \eqref{int rho} nor \eqref{sup rho} on the definition of $\Mr$.
The operator $\Mr$ was studied in \cite{Sawano-Sugano-Tanaka2011} on generalized Morrey spaces.
See also \cite{Ho2016}.
If $\rho(r)=|B(0,r)|^{\alpha/n}$, 
then $M_{\rho}$ is the usual fractional maximal operator $\Ma$.
If $\rho\equiv1$, then $\Mr$ is 
the Hardy-Littlewood maximal operator $M$.
It is known that 
the usual fractional maximal operator $\Ma$ 
is dominated pointwise by 
the fractional integral operator $\Ia$,
that is,
$\Ma f(x)\le C\Ia|f|(x)$ for all $x\in\R^n$.
Then the boundedness of $\Ma$ follows from one of $\Ia$.
However, 
we have a better estimate of $\Mr$ than $\Ir$.

To prove the boundedness of $\Ir$ and $\Mr$ on $\LPp(\R^n)$ and $\wLPp(\R^n)$
we show the pointwise estimate by the Hardy-Littlewood maximal operator $M$
and use the modular inequality of $M$ with respect to Young functions $\Phi$.
The strong and weak type modular inequalities are known. 
In this paper we show the weak-weak type modular inequality.
In general, the modular inequality is stronger than the norm inequality.
For the boundedness on generalized Morrey spaces $L^{(p,\vp)}(\R^n)$
we only need the $L^p$-norm inequality of $M$.
However, for $\LPp(\R^n)$ and $\wLPp(\R^n)$ we need the modular inequality.

The organization of this paper is as follows:
In the next section we state precise definitions of the functions $\Phi$ and $\vp$
by which we define $\LPp(\R^n)$ and $\wLPp(\R^n)$.
Then we state the main results in Section~\ref{sec:main}.
To prove them in the final section,
we give properties of Young functions, Orlicz-Morrey and weak Orlicz-Morrey spaces
in Section~\ref{sec:lemmas}.

At the end of this section, we make some conventions. 
Throughout this paper, we always use $C$ to denote a positive constant 
that is independent of the main parameters involved 
but whose value may differ from line to line.
Constants with subscripts, such as $C_p$, are dependent on the subscripts.
If $f\le Cg$, we then write $f\ls g$ or $g\gs f$; 
and if $f \ls g\ls f$, we then write $f\sim g$.

\section{On the functions $\Phi$ and $\vp$}\label{sec:Young}

In this section 
we state on the functions $\Phi$ and $\vp$ by which we define 
$\LPp(\R^n)$ and $\wLPp(\R^n)$.
We first recall the Young function and its generalization.

For an increasing (i.e. nondecreasing) function 
$\Phi:[0,\infty]\to[0,\infty]$,
let
\begin{equation}\label{aP bP} 
 a(\Phi)=\sup\{t\ge0:\Phi(t)=0\}, \quad 
 b(\Phi)=\inf\{t\ge0:\Phi(t)=\infty\},
\end{equation} 
with convention $\sup\emptyset=0$ and $\inf\emptyset=\infty$.
Then $0\le a(\Phi)\le b(\Phi)\le\infty$.

Let $\biP$ be the set of all increasing functions
$\Phi:[0,\infty]\to[0,\infty]$
such that
\begin{align}\label{ab}
 &0\le a(\Phi)<\infty, \quad 0<b(\Phi)\le\infty, \\
 &\lim_{t\to+0}\Phi(t)=\Phi(0)=0, \label{lim_0} \\
 &\text{$\Phi$ is left continuous on $[0,b(\Phi))$}, \label{left cont} \\
 &\text{if $b(\Phi)=\infty$, then } 
 \lim_{t\to\infty}\Phi(t)=\Phi(\infty)=\infty, \label{left cont infty} \\
 &\text{if $b(\Phi)<\infty$, then } 
 \lim_{t\to b(\Phi)-0}\Phi(t)=\Phi(b(\Phi)) \ (\le\infty). \label{left cont b}
\end{align}

In what follows,
if an increasing and left continuous function $\Phi:[0,\infty)\to[0,\infty)$ satisfies
\eqref{lim_0} and $\dlim_{t\to\infty}\Phi(t)=\infty$,
then we always regard that $\Phi(\infty)=\infty$ and that $\Phi\in\biP$.


For $\Phi\in\biP$,
we recall the generalized inverse of $\Phi$
in the sense of O'Neil \cite[Definition~1.2]{ONeil1965}.

\begin{defn}\label{defn:ginverse}
For $\Phi\in\biP$ and $u\in[0,\infty]$, let
\begin{equation}\label{inverse}
 \Phi^{-1}(u)
 = 
\begin{cases}
 \inf\{t\ge0: \Phi(t)>u\}, & u\in[0,\infty), \\
 \infty, & u=\infty.
\end{cases}
\end{equation}
\end{defn}

Let $\Phi\in\biP$. 
Then $\Phi^{-1}$ is finite, increasing and right continuous on $[0,\infty)$
and positive on $(0,\infty)$.
If $\Phi$ is bijective from $[0,\infty]$ to itself, 
then $\Phi^{-1}$ is the usual inverse function of $\Phi$.
Moreover, if $\Phi\in\biP$, then
\begin{equation}\label{inverse ineq}
 \Phi(\Phi^{-1}(u)) \le u \le  \Phi^{-1}(\Phi(u))
 \quad\text{for all $u\in[0,\infty]$},
\end{equation}
which is a generalization of Property 1.3 in \cite{ONeil1965}.
For its proof see \cite[Proposition~2.2]{Shi-Arai-Nakai2019Taiwan}.

For $\Phi, \Psi\in\biP$, 
we write $\Phi\approx\Psi$
if there exists a positive constant $C$ such that
\begin{equation*} 
     \Phi(C^{-1}t)\le\Psi(t)\le\Phi(Ct)
     \quad\text{for all}\ t\in[0,\infty].
\end{equation*} 
For functions $P,Q:[0,\infty]\to[0,\infty]$, 
we write $P\sim Q$ 
if there exists a positive constant $C$ such that
\begin{equation*} 
     C^{-1}P(t)\le Q(t)\le CP(t)
     \quad\text{for all}\ t\in[0,\infty].
\end{equation*} 
Then, for $\Phi,\Psi\in\biP$, 
\begin{equation}\label{approx equiv}
 \Phi\approx\Psi \quad \Leftrightarrow \quad \Phi^{-1}\sim\Psi^{-1},
\end{equation}
see \cite[Lemma~2.3]{Shi-Arai-Nakai2019Taiwan}.

Now we recall the definition of the Young function and give its generalization.

\begin{defn}\label{defn:Young}
A function $\Phi\in\biP$ is called a Young function 
(or sometimes also called an Orlicz function) 
if $\Phi$ is convex on $[0,b(\Phi))$.
Let $\iPy$ be the set of all Young functions.
Let $\biPy$ be the set of all $\Phi\in\biP$ such that
$\Phi\approx\Psi$ for some $\Psi\in \iPy$.
\end{defn}

For example, 
$\Phi_1$ and $\Phi_2$ defined by \eqref{Young} are in $\iPy$,
and $\Phi_3$ defined by \eqref{gYoung} is in $\biPy\setminus\iPy$.

Similar to Definition~\ref{defn:OM}
we also define $\LPp(\R^n)$ and $\wLPp(\R^n)$ 
by using generalized Young functions $\Phi\in\biPy$
together with $\|\cdot\|_{\Phi,\vp,B}$ and $\|\cdot\|_{\Phi,\vp,B,\weak}$, respectively.
Then $\|\cdot\|_{\Phi,\vp,B}$ and $\|\cdot\|_{\Phi,\vp,B,\weak}$ are quasi norms
and thereby $\LPp(\R^n)$ and $\wLPp(\R^n)$ are quasi Banach spaces.
Note that, for $\Phi,\Psi\in\biPy$, 
if $\Phi\approx\Psi$ and $\vp\sim\psi$,
then
$\LPp(\R^n)=\LPsps(\R^n)$ and $\wLPp(\R^n)=\wLPsps(\R^n)$
with equivalent quasi norms.

\begin{defn}\label{defn:D2 n2}
\begin{enumerate}
\item 
A function $\Phi\in\biP$ is said to satisfy the $\Delta_2$-condition,
denoted by $\Phi\in\bdtwo$, 
if there exists a constant $C>0$ such that
\begin{equation}\label{Delta2}
 \Phi(2t)\le C\Phi(t) 
 \quad\text{for all } t>0.
\end{equation}
\item
A function $\Phi\in\biP$ is said to satisfy the $\nabla_2$-condition,
denoted by $\Phi\in\bntwo$, 
if there exists a constant $k>1$ such that
\begin{equation}\label{nabla2}
 \Phi(t)\le\frac1{2k}\Phi(kt) 
 \quad\text{for all } t>0.
\end{equation}
\item
Let $\Delta_2=\iPy\cap\bdtwo$ and $\nabla_2=\iPy\cap\bntwo$.
\end{enumerate}
\end{defn}

Next, 
we say that a function $\theta:(0,\infty)\to(0,\infty)$ 
satisfies the doubling condition if
there exists a positive constant $C$ such that,
for all $r,s\in(0,\infty)$,
\begin{equation}\label{doubling}
 \frac1C\le\frac{\theta(r)}{\theta(s)}\le C,
 \quad\text{if} \ \ \frac12\le\frac{r}{s}\le2.
\end{equation}
We say that $\theta$ is almost increasing (resp. almost decreasing) if
there exists a positive constant $C$ such that, for all $r,s\in(0,\infty)$,
\begin{equation}\label{almost}
 \theta(r)\le C\theta(s) \quad
 (\text{resp.}\ \theta(s)\le C\theta(r)),
 \quad\text{if $r<s$}.
\end{equation}

In this paper we consider the following class of $\vp:(0,\infty)\to(0,\infty)$.
\begin{defn}\label{defn:cGdec}
Let $\cGdec$ be the set of all functions $\vp:(0,\infty)\to(0,\infty)$
such that 
$\vp$ is almost decreasing
and that
$r\mapsto\vp(r)r^n$ is almost increasing.
That is,
there exists a positive constant $C$ such that, 
for all $r,s\in(0,\infty)$,
\begin{equation*}
 C\vp(r)\ge \vp(s),
 \quad
 \vp(r)r^n\le C\vp(s)s^n, 
 \quad
 \text{if} \ r<s.
\end{equation*}
\end{defn}

If $\vp\in\cGdec$, then $\vp$ satisfies doubling condition.
Let $\psi:(0,\infty)\to(0,\infty)$.
If $\psi\sim\vp$ for some $\vp\in\cGdec$,
then $\psi\in\cGdec$.

\begin{rem}\label{rem:vp bijective}
Let $\vp\in\cGdec$.
Then there exists $\tvp\in\cGdec$ such that $\vp\sim\tvp$
and that 
$\tvp$ is 
continuous and strictly decreasing, 
see \cite[Proposition~3.4]{Nakai2008Studia}.
Moreover,
if 
\begin{equation}\label{cG* dec}
 \lim_{r\to0}\vp(r)=\infty,
 \quad
 \lim_{r\to\infty}\vp(r)=0,
\end{equation}
then $\tvp$ is bijective from $(0,\infty)$ to itself.
\end{rem}

At the end of this section, 
we note that, for $\Phi\in\iPy$ and a ball $B$,
the following relation holds
\begin{equation}\label{weak type} 
     \sup_{t\in(0,\infty)} \Phi(t) m(B,f,t) 
     =
     \sup_{t\in(0,\infty)} t\, m(B,f,\Phi^{-1}(t)) 
     =
     \sup_{t\in(0,\infty)} t\, m(B,\Phi(|f|),t).
\end{equation} 
See \cite{Kawasumi-Nakai2020Hiroshima} for the proof of \eqref{weak type}.
Hence, 
the norm inequality 
$\|Tf\|_{\wLPsp}\le C\|f\|_{\LPp}$ holds
if and only if
\begin{equation*}
 t\,m\left(B,\Psi\left(\frac{|Tf|}{C\|f\|_{\LPp}}\right),t\right)
 \le |B|\vp(r)
\end{equation*}
holds for all balls $B=B(a,r)$ and $t\in(0,\infty)$.

\section{Main results}\label{sec:main}

For a measurable function $f$ and $t\ge0$,
recall that $m(f,t)=|\{x\in\R^n: |f(x)|>t\}|$. 
First we state known modular inequalities for 
the Hardy-Littlewood maximal operator $M$.
For their proofs, 
see \cite[Theorem~1.2.1 and Lemma~1.2.4]{Kokilashvili-Krbec1991} for example. 
\begin{thm}[\cite{Kokilashvili-Krbec1991,Tsereteli1969}]\label{thm:modular}
Let $\Phi\in\biPy$. 
Then there exists a positive constant $C_{\Phi}$ such that 
\begin{equation}\label{weak modular}
 \sup_{t\in(0,\infty)}\Phi(t) m(Mf,t)
 \le
 \int_{\R^n} \Phi(C_{\Phi}|f(x)|)\,dx.
\end{equation}
Moreover, if $\Phi\in\bntwo$, then
\begin{equation}\label{modular}
 \int_{\R^n} \Phi(Mf(x))\,dx
 \le
 \int_{\R^n} \Phi(C_{\Phi}|f(x)|)\,dx.
\end{equation}
\end{thm}

If $\Phi\in\bntwo$ and $\vp\in\cGdec$,
then 
$\wLPp(\R^n)\subset L^1_{\loc}(\R^n)$, see Lemma~\ref{lem:fint_B wLPp} below.
Hence, $Mf$ is well defined for $f\in\wLPp(\R^n)$,
in particular, for $f\in\wLP(\R^n)$.

Our first result is the following modular inequality.

\begin{thm}\label{thm:ww modular}
If $\Phi\in\bntwo$, then
there exists a positive constant $C_{\Phi}$ such that 
\begin{equation}\label{ww modular}
 \sup_{t\in(0,\infty)}\Phi(t) m(Mf,t)
 \le
 \sup_{t\in(0,\infty)}\Phi(t) m(C_{\Phi}f,t).
\end{equation}
\end{thm}

Liu and Wang~\cite{Liu-Wang2013} proved the norm inequality for $\wLP(\R^n)$
with $\Phi\in\dtwo\cap\ntwo$.
Theorem~\ref{thm:ww modular} is its extension.

By Theorems~\ref{thm:modular} and \ref{thm:ww modular}
we will prove the following boundedness,
which is an extension of \cite[Theorem~6.1]{Nakai2008Studia}.

\begin{thm}\label{thm:M}
Let $\Phi\in\biPy$ and $\vp\in\cGdec$.
Then the Hardy-Littlewood maximal operator $M$ 
is bounded from $\LPp(\R^n)$ to $\wLPp(\R^n)$.
Moreover, if $\Phi\in\bntwo$, then
the operator $M$ is bounded from $\LPp(\R^n)$ to itself
and from $\wLPp(\R^n)$ to itself. 
\end{thm}

Next we state the boundedness of generalized fractional integral operators $\Ir$.

\begin{thm}\label{thm:Ir}
Let $\Phi,\Psi\in\biPy$, $\vp\in\cGdec$
and $\rho:(0,\infty)\to(0,\infty)$.
Assume that $\rho$ satisfies \eqref{int rho} and \eqref{sup rho}.
\begin{enumerate}
\item\label{Ir 1}
Assume that $\vp$ satisfies \eqref{cG* dec}
and that 
there exists a positive constant $A$ such that,
for all $r\in(0,\infty)$,
\begin{equation}\label{Ir A}
 \int_0^r\frac{\rho(t)}{t}\,dt\;{\Phi}^{-1}(\vp(r)) 
  +\int_r^{\infty}\frac{\rho(t)\,\Phi^{-1}(\vp(t))}{t}\,dt
 \le
 A\Psi^{-1}(\vp(r)). 
\end{equation}
Then, for any positive constant $C_0$, there exists a positive constant $C_1$ such that, 
for all $f\in \LPp(\R^n)$ with $f\not\equiv0$,
\begin{equation}\label{Ir pointwise}
 \Psi\left(\frac{|\Ir f(x)|}{C_1\|f\|_{\LPp}}\right)
 \le
 \Phi\left(\frac{Mf(x)}{C_0\|f\|_{\LPp}}\right),
 \quad x\in\R^n.
\end{equation}
Consequently, $I_{\rho}$ is bounded 
from $\LPp(\R^n)$ to $\wLPsp(\R^n)$. 
Moreover, if $\Phi\in\bntwo$, then,
for all $f\in \wLPp(\R^n)$ with $f\not\equiv0$,
\begin{equation}\label{Ir pointwise w}
 \Psi\left(\frac{|\Ir f(x)|}{C_1\|f\|_{\wLPp}}\right)
 \le
 \Phi\left(\frac{Mf(x)}{C_0\|f\|_{\wLPp}}\right),
 \quad x\in\R^n.
\end{equation}
Consequently, 
if $\Phi\in\bntwo$, then
$I_{\rho}$ is bounded 
from $\LPp(\R^n)$ to $\LPsp(\R^n)$ by \eqref{Ir pointwise}
and from $\wLPp(\R^n)$ to $\wLPsp(\R^n)$ by \eqref{Ir pointwise w}. 

\item\label{Ir 2}
If $\Ir$ is bounded from $\LPp(\R^n)$ to $\wLPsp(\R^n)$,
then
there exists a positive constant $A'$ such that,
for all $r\in(0,\infty)$,
\begin{equation}\label{Ir A'}
 \int_0^r\frac{\rho(t)}{t}\,dt\;{\Phi}^{-1}(\vp(r)) 
 \le
 A'\Psi^{-1}(\vp(r)). 
\end{equation}
Moreover, 
under the assumption that
there exists a positive constant $C$ such that, for all $r\in(0, \infty)$,
\begin{equation}\label{int vp tn-1}
 \int_0^r \vp(t) t^{n-1}dt \le C \vp(r) r^n,
\end{equation}
if $\Ir$ is bounded from $\LPp(\R^n)$ to $\wLPsp(\R^n)$,
then \eqref{Ir A} holds for some $A\in(0,\infty)$ and for all $r\in(0,\infty)$.

\item\label{Ir 3} 
In the part~\ref{Ir 2} the boundedness from $\LPp(\R^n)$ to $\wLPsp(\R^n)$
can be replaced by the boundedness from $\LPp(\R^n)$ to $L^{(1,\Psi^{-1}(\vp))}(\R^n)$.
\end{enumerate}
\end{thm}

The part~\ref{Ir 1} in the above theorem is 
an extension of \cite[Theorem~7.3]{Nakai2008Studia}.
In the part~\ref{Ir 3}, if $\Psi\in\bntwo$, 
then $\wLPsp(\R^n)\subset L^{(1,\Psi^{-1}(\vp))}(\R^n)$,
see Lemma~\ref{lem:fint_B wLPp} below.
In this case the part~\ref{Ir 3} implies the part~\ref{Ir 2}.

Thirdly, we state the boundedness of the generalized fractional maximal operators $\Mr$.
In this case we need neither the assumption \eqref{int rho} nor \eqref{sup rho}
on the function $\rho:(0,\infty)\to(0,\infty)$.

\begin{thm}\label{thm:Mr}
Let $\Phi,\Psi\in\biPy$, $\vp\in\cGdec$
and $\rho:(0,\infty)\to(0,\infty)$.
\begin{enumerate}
\item\label{Mr 1}
Assume that $\dlim_{r\to\infty}\vp(r)=0$
or that $\Psi^{-1}(t)/\Phi^{-1}(t)$ is almost decreasing on $(0,\infty)$.
If there exists a positive constant $A$ such that,
for all $r\in(0,\infty)$,
\begin{equation}\label{Mr A}
 \left(\sup_{0<t\le r}\rho(t)\right)\Phi^{-1}(\vp(r))
 \le
 A\Psi^{-1}(\vp(r)),
\end{equation}
then, for any positive constant $C_0$, there exists a positive constant $C_1$ such that, 
for all $f\in \LPp(\R^n)$ with $f\not\equiv0$,
\begin{equation}\label{Mr pointwise}
 \Psi\left(\frac{\Mr f(x)}{C_1\|f\|_{\LPp}}\right) 
 \le
 \Phi\left(\frac{Mf(x)}{C_0\|f\|_{\LPp}}\right),
 \quad x\in\R^n.
\end{equation}
Consequently, 
$\Mr$ is bounded from $\LPp(\R^n)$ to $\wLPsp(\R^n)$.
Moreover, if $\Phi\in\bntwo$, 
then, 
for all $f\in \wLPp(\R^n)$ with $f\not\equiv0$,
\begin{equation}\label{Mr pointwise w}
 \Psi\left(\frac{\Mr f(x)}{C_1\|f\|_{\wLPp}}\right) 
 \le
 \Phi\left(\frac{Mf(x)}{C_0\|f\|_{\wLPp}}\right),
 \quad x\in\R^n.
\end{equation}
Consequently, 
if $\Phi\in\bntwo$, then 
$\Mr$ is bounded from $\LPp(\R^n)$ to $\LPsp(\R^n)$ by \eqref{Mr pointwise}
and from $\wLPp(\R^n)$ to $\wLPsp(\R^n)$ by \eqref{Mr pointwise w}.

\item\label{Mr 2} 
If $\Mr$ is bounded from $\LPp(\R^n)$ to $\wLPsp(\R^n)$,
then \eqref{Mr A} holds for some $A\in(0,\infty)$ and for all $r\in(0,\infty)$.

\item\label{Mr 3}
In the part~\ref{Mr 2} the boundedness from $\LPp(\R^n)$ to $\wLPsp(\R^n)$
can be replaced by the boundedness from $\LPp(\R^n)$ to $L^{(1,\Psi^{-1}(\vp))}(\R^n)$.
\end{enumerate}
\end{thm}

The part~\ref{Mr 1} in the above theorem is 
an extension of \cite[Theorem~5.1]{Shi-Arai-Nakai2020Banach}.
If $\Psi\in\bntwo$, 
then the part~\ref{Mr 3} implies the part~\ref{Mr 2} as same as Theorem~\ref{thm:Ir}.

\begin{rem}\label{rem:IrMr}
From \eqref{sup rho} and \eqref{Ir A}
it follows that
\begin{equation*}
 \left(\sup_{0<t\le r}\rho(t)\right)\Phi^{-1}(\vp(r))
 \ls
 \int_0^{K_2r}\frac{\rho(t)}{t}\,dt\, \Phi^{-1}(\vp(r))
 \ls
 \Psi^{-1}(\vp(r)),
\end{equation*}
which is the condition \eqref{Mr A}.
If $\rho(r)=(\log(1/r))^{-\alpha}$ for small $r>0$
or $\rho(r)=(\log r)^{\alpha}$ for large $r>0$ with $\alpha\ge0$,
then the condition \eqref{Mr A} is strictly weaker than \eqref{Ir A}.
\end{rem}


For the case $\vp(r)=1/r^n$, we have the following corollaries.

\begin{cor}[{\cite{Gallardo1988,Krein-Petunin-Semenov1982,Kokilashvili-Krbec1991,Liu-Wang2013}}]\label{cor:M Orl}
Let $\Phi\in\biPy$.
Then the Hardy-Littlewood maximal operator $M$ 
is bounded from $\LP(\R^n)$ to $\wLP(\R^n)$.
Moreover, if $\Phi\in\bntwo$, then
the operator $M$ is bounded from $\LP(\R^n)$ to itself
and from $\wLP(\R^n)$ to itself.
\end{cor}

For the boundedness of the operator $M$ on Orlicz spaces,
see also Cianchi~\cite{Cianchi1999} and Kita~\cite{Kita1996,Kita1997}.
Cianchi~\cite{Cianchi1999} also investigated the operators $\Ma$ and $\Ia$ on Orlicz spaces.

\begin{cor}\label{cor:Ir Orl}
Let $\Phi,\Psi\in\biPy$ and $\rho:(0,\infty)\to(0,\infty)$.
Assume that $\rho$ satisfies \eqref{int rho} and \eqref{sup rho}.
\begin{enumerate}
\item
Assume that 
there exists a positive constant $A$ such that,
for all $r\in(0,\infty)$,
\begin{equation}\label{Ir A Orl}
 \int_0^r\frac{\rho(t)}{t}\,dt\;{\Phi}^{-1}(1/r^n) 
  +\int_r^{\infty}\frac{\rho(t)\,\Phi^{-1}(1/t^n)}{t}\,dt
 \le
 A\Psi^{-1}(1/r^n). 
\end{equation}
Then, for any positive constant $C_0$, there exists a positive constant $C_1$ such that, 
for all $f\in \LP(\R^n)$ with $f\not\equiv0$,
\begin{equation}\label{Ir pointwise Orl}
 \Psi\left(\frac{|\Ir f(x)|}{C_1\|f\|_{\LP}}\right)
 \le
 \Phi\left(\frac{Mf(x)}{C_0\|f\|_{\LP}}\right),
 \quad x\in\R^n.
\end{equation}
Consequently, $I_{\rho}$ is bounded 
from $\LP(\R^n)$ to $\wLPs(\R^n)$. 
Moreover, if $\Phi\in\bntwo$, then,
for all $f\in \wLP(\R^n)$ with $f\not\equiv0$,
\begin{equation}\label{Ir pointwise wOrl}
 \Psi\left(\frac{|\Ir f(x)|}{C_1\|f\|_{\wLP}}\right)
 \le
 \Phi\left(\frac{Mf(x)}{C_0\|f\|_{\wLP}}\right),
 \quad x\in\R^n.
\end{equation}
Consequently, 
if $\Phi\in\bntwo$, then
$I_{\rho}$ is bounded 
from $\LP(\R^n)$ to $\LPs(\R^n)$ 
and from $\wLP(\R^n)$ to $\wLPs(\R^n)$. 

\item\label{Ir 2 Orl}
If $\Ir$ is bounded from $\LP(\R^n)$ to $\wLPs(\R^n)$,
then
there exists a positive constant $A'$ such that,
for all $r\in(0,\infty)$,
\begin{equation}\label{Ir A' Orl}
 \int_0^r\frac{\rho(t)}{t}\,dt\;{\Phi}^{-1}(1/r^n) 
 \le
 A'\Psi^{-1}(1/r^n). 
\end{equation}
Moreover, 
under the assumption that
there exists a positive constant $A''$ such that, for all $r\in(0, \infty)$,
\begin{equation}\label{Ir A'' Orl}
 \int_r^{\infty}\frac{\rho(t)\,\Phi^{-1}(1/t^n)}{t}\,dt
 \le
 A''\Psi^{-1}(1/r^n), 
\end{equation}
if $\Ir$ is bounded from $\LP(\R^n)$ to $\wLPs(\R^n)$,
then
\eqref{Ir A Orl} holds for some $A\in(0,\infty)$ and for all $r\in(0,\infty)$.

\item
In the part~\ref{Ir 2 Orl} the boundedness from $\LP(\R^n)$ to $\wLPs(\R^n)$
can be replaced by the boundedness from $\LP(\R^n)$ to $L^{(1,\psi)}(\R^n)$
with $\psi(r)=\Psi^{-1}(1/r^n)$.
\end{enumerate}
\end{cor}

The above corollary is an extension of 
\cite[Theorem~3]{Deringoz-Guliyev-Nakai-Sawano-Shi2019Posi}.
See \cite{Nakai2001SCMJ},
for examples of $\Phi,\Psi\in\biPy$ 
which satisfy the assumption in Corollary~\ref{cor:Ir Orl}.
The boundedness of the usual fractional integral operator on Orlicz spaces
was given by O'Neil~\cite{ONeil1965}. 
See also \cite{Mizuta-Nakai-Ohno-Shimomura2010JMSJ} for the boundedness of $\Ir$ 
on Orlicz space $\LP(\Omega)$ with bounded domain $\Omega\subset\R^n$.

\begin{cor}\label{cor:Mr Orl}
Let $\Phi,\Psi\in\biPy$ and $\rho:(0,\infty)\to(0,\infty)$.
\begin{enumerate}
\item
If there exists a positive constant $A$ such that,
for all $r\in(0,\infty)$,
\begin{equation}\label{Mr A PPs}
 \left(\sup_{0<t\le r}\rho(t)\right)\Phi^{-1}(1/r^n)
 \le
 A\Psi^{-1}(1/r^n),
\end{equation}
then, for any positive constant $C_0$, there exists a positive constant $C_1$ such that, 
for all $f\in \LP(\R^n)$ with $f\not\equiv0$,
\begin{equation}\label{Mr pointwise Orl}
 \Psi\left(\frac{\Mr f(x)}{C_1\|f\|_{\LP}}\right) 
 \le
 \Phi\left(\frac{Mf(x)}{C_0\|f\|_{\LP}}\right),
 \quad x\in\R^n.
\end{equation}
Consequently, 
$\Mr$ is bounded from $\LP(\R^n)$ to $\wLPs(\R^n)$.
Moreover, if $\Phi\in\bntwo$, 
then, 
for all $f\in \wLP(\R^n)$ with $f\not\equiv0$,
\begin{equation}\label{Mr pointwise wOrl}
 \Psi\left(\frac{\Mr f(x)}{C_1\|f\|_{\wLP}}\right) 
 \le
 \Phi\left(\frac{Mf(x)}{C_0\|f\|_{\wLP}}\right),
 \quad x\in\R^n.
\end{equation}
Consequently, 
if $\Phi\in\bntwo$, then 
$\Mr$ is bounded from $\LP(\R^n)$ to $\LPs(\R^n)$
and from $\wLP(\R^n)$ to $\wLPs(\R^n)$.

\item\label{Mr 2 Orl}
If $\Mr$ is bounded from $\LP(\R^n)$ to $\wLPs(\R^n)$,
then \eqref{Mr A PPs} holds for some $A\in(0,\infty)$ and for all $r\in(0,\infty)$.

\item
In the part~\ref{Mr 2 Orl} the boundedness from $\LP(\R^n)$ to $\wLPs(\R^n)$
can be replaced by the boundedness from $\LP(\R^n)$ to $L^{(1,\psi)}(\R^n)$
with $\psi(r)=\Psi^{-1}(1/r^n)$.
\end{enumerate}
\end{cor}

The above corollary is an extension of 
\cite[Theorem~3.8]{Shi-Arai-Nakai2019Taiwan}.

For the case $\Phi(t)=t^p$, we have the following corollaries.

\begin{cor}\label{cor:M}
Let $1\le p<\infty$ and $\vp\in\cGdec$.
Then the Hardy-Littlewood maximal operator $M$ 
is bounded from $L^{(p,\vp)}(\R^n)$ to $\wL^{(p,\vp)}(\R^n)$.
Moreover, if $1<p<\infty$, then the operator $M$ is bounded from $L^{(p,\vp)}(\R^n)$ to itself
and from $\wL^{(p,\vp)}(\R^n)$ to itself.
\end{cor}

The above corollary is an extension of 
\cite[Theorem~1]{Nakai1994MathNachr} and \cite[Corollary~6.2]{Nakai2008Studia}.
The boundedness of the Hardy-Littlewood maximal operator on the classical Morrey space
was proven by Chiarenza and Frasca~\cite{Chiarenza-Frasca1987}.

\begin{cor}\label{cor:Ir}
Let $1\le p<q<\infty$, $\vp\in\cGdec$
and $\rho:(0,\infty)\to(0,\infty)$.
Assume that $\rho$ satisfies \eqref{int rho} and \eqref{sup rho}.
\begin{enumerate}
\item
Assume that $\vp$ satisfies \eqref{cG* dec}
and that 
there exists a positive constant $A$ such that,
for all $r\in(0,\infty)$,
\begin{equation}\label{Ir A pq}
 \int_0^r\frac{\rho(t)}{t}\,dt\,\vp(r)^{1/p} 
  +\int_r^{\infty}\frac{\rho(t)\vp(t)^{1/p}}{t}\,dt
 \le
 A\vp(r)^{1/q}. 
\end{equation}
Then there exists a positive constant $C$ such that, 
for all $f\in L^{(p,\vp)}(\R^n)$,
\begin{equation}\label{Ir pointwise Mor}
 |\Ir f(x)|
 \le
 C
 Mf(x)^{p/q}\left(\|f\|_{L^{(p,\vp)}}\right)^{1-p/q},
 \quad x\in\R^n.
\end{equation}
Consequently, $I_{\rho}$ is bounded 
from $L^{(p,\vp)}(\R^n)$ to $\wL^{(q,\vp)}(\R^n)$.
Moreover, if $p>1$, then,
for all $f\in \wL^{(p,\vp)}(\R^n)$,
\begin{equation}\label{Ir pointwise wMor}
 |\Ir f(x)|
 \le
 C
 Mf(x)^{p/q}\left(\|f\|_{\wL^{(p,\vp)}}\right)^{1-p/q},
 \quad x\in\R^n.
\end{equation}
Consequently, if $p>1$, then
$I_{\rho}$ is bounded 
from $L^{(p,\vp)}(\R^n)$ to $L^{(q,\vp)}(\R^n)$ by \eqref{Ir pointwise Mor}
and 
from $\wL^{(p,\vp)}(\R^n)$ to $\wL^{(q,\vp)}(\R^n)$ by \eqref{Ir pointwise wMor}.

\item\label{Ir 2 Mor}
If $\Ir$ is bounded 
from $L^{(p,\vp)}(\R^n)$ to $\wL^{(q,\vp)}(\R^n)$,
then
there exists a positive constant $A'$ such that,
for all $r\in(0,\infty)$,
\begin{equation}\label{Ir A' pq}
 \int_0^r\frac{\rho(t)}{t}\,dt\,\vp(r)^{1/p} 
 \le
 A'\vp(r)^{1/q}. 
\end{equation}
Moreover, 
under the assumption \eqref{int vp tn-1},
if $\Ir$ is bounded 
from $L^{(p,\vp)}(\R^n)$ to $\wL^{(q,\vp)}(\R^n)$,
then \eqref{Ir A pq} holds for some $A\in(0,\infty)$ and for all $r\in(0,\infty)$.

\item
In the part~\ref{Ir 2 Mor} the boundedness from $L^{(p,\vp)}(\R^n)$ to $\wL^{(q,\vp)}(\R^n)$
can be replaced by the boundedness from $L^{(p,\vp)}(\R^n)$ to $L^{(1,\vp^{1/q})}(\R^n)$.
\end{enumerate}
\end{cor}

The above corollary is an extension of 
\cite[Theorems~1.1, 1.2, 1.8 and 1.9]{Eridani-Gunawan-Nakai-Sawano2014MIA}. 
The condition \eqref{Ir A pq} was first given by Gunawan~\cite{Gunawan2003}.
The boundedness of the usual fractional integral operator on the classical Morrey space
was proven by Adams~\cite{Adams1975}.

\begin{cor}\label{cor:Mr}
Let $1\le p\le q<\infty, \vp\in\cGdec$
and $\rho:(0,\infty)\to(0,\infty)$.
\begin{enumerate}
\item
If there exists a positive constant $A$ such that,
for all $r\in(0,\infty)$,
\begin{equation}\label{Mr A pq}
 \left(\sup_{0<t\le r}\rho(t)\right)\vp(r)^{1/p}
 \le
 A\vp(r)^{1/q},
\end{equation}
then there exists a positive constant $C$ such that, 
for all $f\in L^{(p,\vp)}(\R^n)$,
\begin{equation}\label{Mr pointwise Mor}
 \Mr f(x)
 \le
 C
 Mf(x)^{p/q}\left(\|f\|_{L^{(p,\vp)}}\right)^{1-p/q},
 \quad x\in\R^n.
\end{equation}
Consequently, 
$\Mr$ is bounded from $L^{(p,\vp)}(\R^n)$ to $\wL^{(q,\vp)}(\R^n)$.
Moreover, 
if $p>1$, then,
for all $f\in \wL^{(p,\vp)}(\R^n)$,
\begin{equation}\label{Mr pointwise wMor}
 \Mr f(x)
 \le
 C
 Mf(x)^{p/q}\left(\|f\|_{\wL^{(p,\vp)}}\right)^{1-p/q},
 \quad x\in\R^n.
\end{equation}
Consequently, 
if $p>1$, then 
$\Mr$ is bounded from $L^{(p,\vp)}(\R^n)$ to $L^{(q,\vp)}(\R^n)$ by \eqref{Mr pointwise Mor}
and 
from $\wL^{(p,\vp)}(\R^n)$ to $\wL^{(q,\vp)}(\R^n)$ by \eqref{Mr pointwise wMor}.

\item\label{Mr 2 Mor} 
If $\Mr$ is bounded from $L^{(p,\vp)}(\R^n)$ to $\wL^{(q,\vp)}(\R^n)$, 
then \eqref{Mr A pq} holds for some $A\in(0,\infty)$ and for all $r\in(0,\infty)$.

\item
In the part~\ref{Mr 2 Mor} the boundedness from $L^{(p,\vp)}(\R^n)$ to $\wL^{(q,\vp)}(\R^n)$
can be replaced by the boundedness from $L^{(p,\vp)}(\R^n)$ to $L^{(1,\vp^{1/q})}(\R^n)$.
\end{enumerate}
\end{cor}

\section{Properties of Young functions and Orlicz-Morrey spaces}\label{sec:lemmas}

In this section we state the properties of Young functions, 
Orlicz-Morrey and weak Orlicz Morrey spaces.
By the convexity, 
any Young function $\Phi$ is continuous on $[0,b(\Phi))$ 
and strictly increasing on $[a(\Phi),b(\Phi)]$.
Hence $\Phi$ is bijective from $[a(\Phi),b(\Phi)]$ to $[0,\Phi(b(\Phi))]$.

\begin{defn}\label{defn:iPy}
Let 
\begin{align*}
 \cY^{(1)}
 &=
 \left\{\Phi\in\iPy:b(\Phi)=\infty\right\}, 
\\
 \cY^{(2)}
 &=
 \left\{\Phi\in\iPy:b(\Phi)<\infty,\ \Phi(b(\Phi))=\infty\right\}, 
\\
 \cY^{(3)}
 &=
 \left\{\Phi\in\iPy:b(\Phi)<\infty,\ \Phi(b(\Phi))<\infty\right\}.
\end{align*}
\end{defn}

\begin{rem}\label{rem:D2 n2}
\begin{enumerate}
\item\label{Y1Y2}
If $\Phi\in\cY^{(1)}\cup\cY^{(2)}$, then
$\Phi(\Phi^{-1}(u))=u$ for all $u\in[0,\infty]$.

\item\label{Y3}
If $\Phi\in\cY^{(3)}$ and $0<\delta<1$, 
then there exists a Young function $\Psi\in\cY^{(2)}$ such that $b(\Phi)=b(\Psi)$ and
$$
 \Psi(\delta t) \le \Phi(t) \le \Psi(t) \quad\text{for all} \ t\in[0,\infty).
$$
To see this we only set $\Psi=\Phi+\Theta$, where
we choose $\Theta\in\cY^{(2)}$
such that $a(\Theta)=\delta\,b(\Phi)$ and $b(\Theta)=b(\Phi)$.

\item\label{bntwo} 
$\bntwo\subset\biPy$ (\cite[Lemma~1.2.3]{Kokilashvili-Krbec1991}).

\item\label{D2 n2 approx}
Let $\Phi\in\biPy$.
Then
$\Phi\in\bdtwo$ if and only if $\Phi\approx\Psi$ for some $\Psi\in\dtwo$,
and,  
$\Phi\in\bntwo$ if and only if $\Phi\approx\Psi$ for some $\Psi\in\ntwo$.

\item\label{Phi-1}
If $\Phi\in\iPy$,
then
$\Phi^{-1}$ satisfies the doubling condition by its concavity.

\item\label{Phi/tp inc}
Let $\Phi\in\iPy$.
Then $\Phi\in\ntwo$ if and only if 
$t\mapsto\dfrac{\Phi(t)}{t^p}$ is almost increasing for some $p\in(1,\infty)$.
\end{enumerate}
\end{rem}


For a Young function $\Phi$, 
its complementary function is defined by
\begin{equation*}
\cPhi(t)= 
\begin{cases}
   \sup\{tu-\Phi(u):u\in[0,\infty)\}, & t\in[0,\infty), \\
   \infty, & t=\infty.
 \end{cases}
\end{equation*}
Then $\cPhi$ is also a Young function,
and $(\Phi,\cPhi)$ is called a complementary pair.
For example, 
if $\Phi(t)=t^p/p$, then $\cPhi(t)=t^{p'}/p'$
for $p,p'\in(1,\infty)$ and $1/p+1/p'=1$.
If $\Phi(t)=t$, then
\begin{equation*}
 \cPhi(t)=
\begin{cases}
 0, & t\in[0,1], \\
 \infty, & t\in(1,\infty].
\end{cases}
\end{equation*}

Let $(\Phi,\cPhi)$ be a complementary pair of functions in $\iPy$. 
Then the following inequality holds:
\begin{equation}\label{Phi cPhi r}
 t\le\Phi^{-1}(t) \cPhi^{-1}(t)\le2t
 \quad\text{for}\quad t\in[0,\infty],
\end{equation}
which is (1.3) in \cite{Torchinsky1976}.


Orlicz and weak Orlicz spaces on a measure space $(\Omega,\mu)$
defined by the following:
For $\Phi\in\biPy$, 
let $\LP(\Omega,\mu)$ and $\wLP(\Omega,\mu)$ 
be the set of all measurable functions $f$
such that the following functionals are finite, respectively:
\begin{align*}
  \|f\|_{\LP(\Omega,\mu)} &=
  \inf\left\{ \lambda>0: 
    \int_{\Omega} \Phi\left(\frac{|f(x)|}{\lambda}\right) d\mu(x)
      \le 1
      \right\}, 
\\
  \|f\|_{\wLP(\Omega,\mu)} &=
  \inf\left\{ \lambda>0: 
    \sup_{t\in(0,\infty)}\Phi(t)\,\mu\!\left(\frac{f}{\lambda}, t\right)
      \le 1
      \right\},
\end{align*}
where 
\begin{equation*} 
\mu(f,t)=\mu\big(\{x\in \Omega: |f(x)|>t\}\big).
\end{equation*}
Then 
$\|\cdot\|_{\LP(\Omega,\mu)}$ and $\|\cdot\|_{\wLP(\Omega,\mu)}$ are quasi norms
and thereby $\LP(\Omega,\mu)$ and $\wLP(\Omega,\mu)$ are quasi Banach spaces.
If $\Phi\in\iPy$, then $\|\cdot\|_{\LP(\Omega,\mu)}$ is a norm 
and thereby $\LP(\Omega,\mu)$ is a Banach space.
For $\Phi,\Psi\in\biPy$,
if $\Phi\approx\Psi$, then $\LP(\Omega,\mu)=\LPs(\Omega,\mu)$ 
and $\wLP(\Omega,\mu)=\wLPs(\Omega,\mu)$
with equivalent quasi norms, respectively.

For a measurable set $G\subset\Omega$, let $\chi_G$ be its characteristic function.
For $\Phi\in\iPy$, 
it is known that
\begin{equation}\label{chi Orlicz norm}
 \|\chi_G\|_{\wLP(\Omega,\mu)}
 =
 \|\chi_G\|_{\LP(\Omega,\mu)}
 =
 \frac1{\Phi^{-1}(1/\mu(G))},
 \quad\text{if} \ \mu(G)>0.
\end{equation}
Let $(\Phi,\cPhi)$ be a complementary pair of functions in $\iPy$. 
Then the following generalized H\"older's inequality holds:
\begin{equation}\label{g Holder}
 \int_{\Omega} |f(x)g(x)| \,d\mu(x) \le 2\|f\|_{\LP(\Omega,\mu)} \|g\|_{\LcP(\Omega,\mu)}.
\end{equation}

For a ball $B=B(a,r)\subset\R^n$, let $\mu_B=dx/(|B|\vp(r))$.
Then we have the following relations.
\begin{equation}\label{LP B norm}
 \|f\|_{\Phi,\vp,B}=\|f\|_{\LP(B,\,\mu_B)},
\quad
 \|f\|_{\Phi,\vp,B,\weak}=\|f\|_{\wLP(B,\,\mu_B)}.
\end{equation}
By the relations \eqref{chi Orlicz norm} and \eqref{LP B norm}, we have
\begin{equation}\label{chi norm B}
 \|\chi_B\|_{\Phi,\vp,B,\weak}
 =
 \|\chi_B\|_{\Phi,\vp,B}
 =
 \frac1{\Phi^{-1}(1/\mu_B(B))}
 =
 \frac1{\Phi^{-1}(\vp(r))}.
\end{equation}
For functions $f,g$ on $\R^n$, 
by \eqref{g Holder} and \eqref{LP B norm}, we have
\begin{equation}\label{g Holder B}
 \frac1{|B|\vp(r)} \int_{B} |f(x)g(x)| \,dx \le 2\|f\|_{\Phi,\vp,B} \|g\|_{\cPhi,\vp,B}.
\end{equation}


\begin{lem}\label{lem:chi norm 2}
Let $\Phi\in\iPy$ and $\vp \in \cGdec$.
Then
there exists a constant $C\ge1$ such that,
for any ball $B=B(a,r)$, 
\begin{equation}\label{chi norm}
 \frac1{\Phi^{-1}(\vp(r))}
 \le
 \|\chi_B\|_{\wLPp}
 \le
 \|\chi_B\|_{\LPp}
 \le
 \frac{C}{\Phi^{-1}(\vp(r))}.
\end{equation}
\end{lem}

\begin{proof}
By \eqref{chi norm B} we have
\begin{equation*}
 \frac1{\Phi^{-1}(\vp(r))}
 =
 \|\chi_B\|_{\Phi,\vp,B,\weak}
 \le
 \|\chi_B\|_{\wLPp}
 \le
 \|\chi_B\|_{\LPp}.
\end{equation*}
The last inequality in \eqref{chi norm} has been proven 
in \cite[Lemma~4.1]{Shi-Arai-Nakai2020Banach}. 
\end{proof}

\begin{cor}\label{cor:chi norm}
Let $\vp\in\cGdec$.
Then
there exists a constant $C\ge1$ such that,
for any ball $B=B(a,r)$, 
\begin{equation*}
 \frac1{\vp(r)}
 \le
 \|\chi_B\|_{\wL^{(1,\vp)}}
 \le
 \|\chi_B\|_{L^{(1,\vp)}}
 \le
 \frac{C}{\vp(r)}.
\end{equation*}
\end{cor}

The following relation between Orlicz-Morrey spaces and generalized Morrey spaces is known:
\begin{lem}[{\cite[Corollary~4.7]{Nakai2008Studia}}]\label{lem:fint_B f}
Let $\Phi\in\iPy$ and $\vp:(0,\infty)\to(0,\infty)$.
Then there exists a positive constant $C$ such that, 
for all $f\in\LPp(\R^n)$ and for all balls $B=B(a,r)$,
\begin{equation}\label{fint PpB}
 \fint_{B}|f(x)|\,dx
 \le
 2\Phi^{-1}(\vp(r)) \|f\|_{\Phi,\vp,B}.
\end{equation}
Consequently, 
$\LPp(\R^n)$ continuously embeds in $L^{(1,\Phi^{-1}(\vp))}(\R^n)$.
\end{lem}

In the above lemma we do not need to assume $\vp\in\cGdec$, while it was assumed in \cite[Corollary~4.7]{Nakai2008Studia}.
Actually, we can prove \eqref{fint PpB} by \eqref{Phi cPhi r}, \eqref{chi norm B} and \eqref{g Holder B}. 

\begin{rem}\label{rem:wLPp}
If $\Phi\in\iPy$ and $\vp\in\cGdec$, 
then $\Phi^{-1}(\vp)\in\cGdec$,
see \cite[Corollary~4.7]{Nakai2008Studia}.
\end{rem}

The following lemma shows that any weak Orlicz-Morrey space continuously embeds 
in a generalized Morrey space,
whenever $\Phi\in\bntwo$,
and it is an extension of \cite[Theorem~3.4]{Nakai2002Lund} which treats weak Orlicz spaces.
Note that if $\Phi\in\bntwo$, then $\Phi\approx\Phi_1$ for some $\Phi_1\in\ntwo$,
see Remark~\ref{rem:D2 n2}~\ref{bntwo} and \ref{D2 n2 approx}.
In this case $\|f\|_{\Phi,\vp,B,\weak}\sim\|f\|_{\Phi_1,\vp,B,\weak}$ and $\Phi^{-1}\sim\Phi_1^{-1}$.

\begin{lem}\label{lem:fint_B wLPp}
Let $\Phi\in\nabla_2$ and $\vp:(0,\infty)\to(0,\infty)$.
Then there exists a positive constant $C$ such that, 
for all $f\in\wLPp(\R^n)$ and for all balls $B=B(a,r)$,
\begin{equation*}
 \fint_{B}|f(x)|\,dx
 \le
 C\Phi^{-1}(\vp(r)) \|f\|_{\Phi,\vp,B,\weak}.
\end{equation*}
Consequently, 
$\wLPp(\R^n)$ continuously embeds in $L^{(1,\Phi^{-1}(\vp))}(\R^n)$.
\end{lem}

\begin{proof}
Case 1: $\Phi\in\cY^{(1)} \cup \cY^{(2)}$.
We may assume that $\|f\|_{\Phi,\vp,B,\weak}=1$.
If $b(\Phi)<\infty$ and $t\in[b(\Phi),\infty)$,
then $m(f,t)=0$.
For any ball $B$ of radius $r$, let $t_0=\Phi^{-1}(\vp(r))$.
Then $\Phi(t_0)=\vp(r)\in(0,\infty)$ by Remark~\ref{rem:D2 n2}~\ref{Y1Y2}.
That is, $t_0\in(a(\Phi),b(\Phi))$.
Take $p\in(1,\infty)$ such that $\Phi(t)/t^p$ is almost increasing,
see Remark~\ref{rem:D2 n2}~\ref{Phi/tp inc}.
Then
\begin{align*}
 \int_B|f(x)|\,dx
 &=
 \int_0^{t_0} m(B,f,t) \,dt
 +\int_{t_0}^{b(\Phi)} m(B,f,t) \,dt
\\
 &\le
 {t_0}|B|
 + \int_{t_0}^{b(\Phi)} \frac{ \vp(r) |B| }{ \Phi(t) } \,dt
\\
 &=
 {t_0}|B|
 + \vp(r)|B| \int_{t_0}^{b(\Phi)} \frac{ t^p }{ \Phi(t) } t^{-p} \,dt
\\
 &\ls
 {t_0}|B|
 + \vp(r)|B| \frac{ {t_0}^p }{ \Phi({t_0}) } \int_{t_0}^{b(\Phi)} t^{-p} \,dt
\\
 &=
 {t_0}|B|
 +\frac{{t_0}{|B|\vp(r)} }{(p-1)\Phi({t_0})} 
\\
 &=
 {t_0}|B|
 +\frac{{t_0}|B|}{(p-1)}. 
\end{align*}
This shows the conclusion.

Case 2: $\Phi\in\cY^{(3)}$.
In this case, for any $\delta\in(0,1)$, 
there exists $\Phi_1\in\cY^{(2)}$ such that 
\begin{equation*}
 \Phi_1(\delta t)\le\Phi(t)\le\Phi_1(t), \quad t\in[0,\infty],
\end{equation*}
see Remark~\ref{rem:D2 n2}~\ref{Y3}.
It follows that
\begin{equation*}
 \delta\Phi^{-1}(u)\le\Phi_1^{-1}(u)\le\Phi^{-1}(u) 
 \quad\text{and}\quad
 \delta\|f\|_{\wL^{(\Phi_1,\vp)}}\le\|f\|_{\wLPp}\le\|f\|_{\wL^{(\Phi_1,\vp)}}
\end{equation*}
By Case 1 we have
\begin{equation*}
 \fint_{B}|f(x)|\,dx
 \le
 C\Phi_1^{-1}(\vp(r)) \|f\|_{\wL^{(\Phi_1,\vp)}}
 \le
 C\Phi^{-1}(\vp(r)) \|f\|_{\wL^{(\Phi,\vp)}}/\delta.
\end{equation*}
Letting $\delta\to1$, we have the conclusion.
\end{proof}

\begin{lem}[{\cite[Lemma~9.4]{Nakai2008Studia}}]\label{lem:Nakai2008Studia9.4}
Let $\Phi\in\iPy$ and $\vp \in \cGdec$, and let $f \in \LPp(\R^n)$. 
For a ball $B=B(a,r)$, if $\supp f \cap 2B = \emptyset$,
then
\begin{equation*}
 Mf(x) \le C\Phi^{-1}(\vp(r)) \|f\|_{\LPp}
 \quad\text{for}\quad x \in B, 
\end{equation*}
where the constant $C$ depends only on $\Phi$ and $\vp$.
\end{lem}

\begin{lem}\label{lem:Mf wLPp}

Let $\Phi\in\ntwo$ and $\vp\in\cGdec$, and let $f \in \LPp(\R^n)$.
For a ball $B=B(a,r)$, if $\supp f \cap 2B = \emptyset$,
then
\begin{equation*}
 Mf(x) \le C\Phi^{-1}\left( \vp(r) \right)\|f\|_{\wLPp}
 \quad\text{for}\quad x \in B, 
\end{equation*}
where the constant $C$ depends only on $\Phi$ and $\vp$.
\end{lem}

\begin{proof}
For any ball $B' \ni x$ whose radius is $s$, 
if $s\le r/2$, then $\int_{B'} |f(x)| dx = 0$,
and, if $s>r/2$, 
then, using Lemma~\ref{lem:fint_B wLPp}, 
we have
\begin{align*}
 \fint_{B'} |f(x)|\,dx
 &\ls
 \Phi^{-1}(\vp(s))\|f\|_{\wLPp}
 \\
 &\ls
 \Phi^{-1}(\vp(r))\|f\|_{\wLPp},
\end{align*}
since $r\mapsto\Phi^{-1}(\vp(r))$ is almost decreasing and satisfies the doubling condition.
\end{proof}

\section{Proofs}\label{sec:proof}

We first note that, to prove the theorems,
we may assume that $\Phi,\Psi\in\iPy$ instead of  $\Phi,\Psi\in\biPy$.
For example, if $\Phi$ and $\Psi$ satisfy \eqref{Ir A}
and $\Phi\approx\Phi_1$, $\Psi\approx\Psi_1$, then 
$\Phi_1$ and $\Psi_1$ also satisfy \eqref{Ir A} by the relation \eqref{approx equiv}.
Moreover,
\begin{equation*}
\LPp(\R^n)=L^{(\Phi_1,\vp)}(\R^n),
 \quad
 \LPsp(\R^n)=L^{(\Psi_1,\vp)}(\R^n),
\end{equation*}
and
\begin{equation*}
\wLPp(\R^n)=\wL^{(\Phi_1,\vp)}(\R^n),  \quad \wLPsp(\R^n)=\wL^{(\Psi_1,\vp)}(\R^n),
\end{equation*}
with equivalent quasi norms.
Similarly, by Remark~\ref{rem:D2 n2}~\ref{D2 n2 approx} 
we may assume that $\Phi\in\ntwo$ instead of $\Phi\in\bntwo$.
By Remark~\ref{rem:vp bijective}
we may also assume that $\vp$ is continuous and strictly decreasing.


\subsection{Proof of Theorem~\ref{thm:ww modular}}\label{ss:proof thm ww modular}

To prove Theorem~\ref{thm:ww modular}
we use the following lemma.

\begin{lem}[{\cite[page~92]{Torchinsky1986}}]\label{lem:Mf t}
If $f\in L^1(\R^n)$, then
\begin{equation*}
 m(Mf,t)\le\frac{C}t \int_{|f|>t/2}|f(x)|\,dx,
 \quad t\in(0,\infty).
\end{equation*}
\end{lem}

\begin{proof}[\bf Proof of Theorem~\ref{thm:ww modular}]
We may assume that $\Phi\in\ntwo$ and $f\ge0$.
By Remark~\ref{rem:D2 n2}~\ref{Phi/tp inc} we can take $p>1$ such that
$\Phi(t)/t^p$ is almost increasing.
Then, 
\begin{equation*}
 \frac{\Phi(t)}{t^p}\le C_p\frac{\Phi(s)}{s^p}, \quad t<s.
\end{equation*}
Let 
\begin{equation*}
 N(f)=\sup_{t\in(0,\infty)}\Phi(t)m(2f, t)<\infty
 \quad\text{and}\quad G_t=\{2f>t\},
\end{equation*}
and let $g_t(x)=(f(x)-t/2)\chi_{G_t}(x)$ and $h_t(x)=f(x)-g_t(x)$.
Then $Mh_t(x)\le t/2$, which implies
\begin{equation*}
 m(Mf,t)
 \le
 m(Mg_t,t/2)+m(Mh_t,t/2)
 =
 m(Mg_t,t/2).
\end{equation*}
We consider three cases
$t\in(0,a(\Phi)]$, $t\in(a(\Phi),b(\Phi))$ and $t\in[b(\Phi),\infty)$.

Case 1:
If  $a(\Phi)>0$ and $t\in (0, a(\Phi)]$, 
then $\Phi(t) m(Mf,t) = 0$.

Case 2:
If $0\le a(\Phi) < b(\Phi) \le \infty$ and $t\in(a(\Phi),b(\Phi))$,
then $|G_t|=m(2f,t)<\infty$. 
Hence
\begin{align*}
 2\int_{\R^n}g_t(x)\,dx
 &=
 \int_{G_t}(2f(x)-t)\,dx
\\
 &=
 \int_0^t m(G_t,2f,s)\,ds + \int_t^{b(\Phi)} m(G_t,2f,s)\,ds - t|G_t|
\\
 &\le
 \int_t^{b(\Phi)} m(2f,s)\,ds
 =
 \frac1{\Phi(t)}\int_t^{b(\Phi)}\frac{\Phi(t)}{\Phi(s)}\Phi(s)m(2f,s)\,ds
\\
 &\le
 \frac{C_pN(f)}{\Phi(t)}\int_t^{b(\Phi)}\left(\frac{t}{s}\right)^{p}\,ds
 \le
 \frac{C_pN(f)t}{(p-1)\Phi(t)}<\infty.
\end{align*}
That is, $g_t\in L^1(\R^n)$ for $t\in(a(\Phi),b(\Phi))$.
By Lemma~\ref{lem:Mf t} we have
\begin{equation*}
 \Phi(t)m(Mf,t)
 \le
 \Phi(t)m(Mg_t,t/2)
 \le
 \frac{2C\Phi(t)}{t}\int_{\R^n}g_t(x)\,dx
 \le
 \frac{CC_p}{p-1}N(f).
\end{equation*}

Case 3:
If $b(\Phi)<\infty$ and $t\in[b(\Phi),\infty)$, then $|G_t|=m(2f,t)=0$ and $g_t=0$ a.e.
Hence $Mg_t=0$ and
\begin{equation*}
 \Phi(t)m(Mf,t)
 \le
 \Phi(t)m(Mg_t,t/2)
 =0.
\end{equation*}

In conclusion, for some $C'\ge1$ and all $t\in(0,\infty)$,
\begin{align*}
 \Phi(t) m(Mf,t) 
 &\le 
 C'N(f)
\\
 &\le 
 \sup_{s\in(0,\infty)}\Phi(C's)m(2f,s)
\\
 &=
 \sup_{s\in(0,\infty)}\Phi(s)m(2C'f,s).
\end{align*}
This shows the conclusion.
\end{proof}

\subsection{Proof of Theorem~\ref{thm:M}}\label{ss:proof thm M}

In this subsection we prove Theorem~\ref{thm:M}
by using Theorems~\ref{thm:modular} and \ref{thm:ww modular}.

\begin{proof}[\bf Proof of Theorem~\ref{thm:M}]
We may assume that $\Phi\in\iPy$.
Let $f \in \LPp(\R^n)$ and $||f||_{\LPp}=1$. 
To prove the norm inequality $\|Mf\|_{\wLPp}\le C_0$,
it is enough to prove that,
for any ball $B=B(a,r)$,
\begin{equation}\label{M weak}
 \|Mf\|_{\Phi,\vp,B,\weak}
 \le
 C_0.
\end{equation}

Let $f = f_1 + f_2$, $f_1 = f \chi_{2B}$.
Then we have 
\begin{equation}\label{int Phi(f1)}
 \int_{\R^n} \Phi(|f_1(x)|)\,dx
 \le
 \int_{2B} \Phi\left(\frac{|f(x)|}{\|f\|_{\Phi,\vp,2B}}\right)\,dx
 \le 
 |2B| \vp(2r)
 \le
 C_{n,\vp}|B|\vp (r),
\end{equation}
for some constant $C_{n,\vp}\ge 1$ by the doubling condition of $\vp$. 
Hence, \eqref{weak modular} 
yields
\begin{align*}
 \sup_{t\in(0,\infty)}\Phi(t)\ m\left(B,\frac{Mf_1}{C_{n,\vp}C_{\Phi}},t\right)
 &\le
 \frac1{C_{n,\vp}}
 \sup_{t\in(0,\infty)}\Phi(t)\ m\left(\frac{Mf_1}{C_{\Phi}},t\right) \\
 &\le
 \frac1{C_{n,\vp}}
 \int_{\R^n} \Phi(|f_1(x)|)\,dx
 \le
 |B|\vp (r).
\end{align*}
Next, 
since $\supp f_{2} \cap 2B = \emptyset$, 
by Lemma~\ref{lem:Nakai2008Studia9.4}
we have $Mf_2(x)\le C\Phi^{-1}(\vp(r))$ for $x\in B$.
Then
\begin{equation*}
 \sup_{t\in(0,\infty)}
 \Phi(t)\ m\left(B,\frac{Mf_{2}}{C},t\right) 
 \le
 \int_{B} \Phi \left( \frac{Mf_{2}(x)}{C} \right) dx
 \le
 \int_{B} \vp(r)dx = |B|\vp(r).
\end{equation*}
In the above we use \eqref{inverse ineq} for the second inequality.
Therefore,
\begin{equation*} 
 \|Mf\|_{\Phi,\vp,B,\weak}
 \le
 2\left(\|Mf_1\|_{\Phi,\vp,B,\weak}+\|Mf_2\|_{\Phi,\vp,B,\weak}\right)
 \le
 2(C_{n,\vp}C_{\Phi}+C),
\end{equation*}
which shows \eqref{M weak}.
If $\Phi\in\ntwo$, then by the same way we have 
the norm inequality $\|Mf\|_{\LPp}\le C_0$, 
using \eqref{modular} instead of \eqref{weak modular}.

Next, let $f \in \wLPp(\R^n)$ and $||f||_{\wLPp}=1$. 
To prove the norm inequality $\|Mf\|_{\wLPp}\le C_0$,
it is enough to prove \eqref{M weak}
for any ball $B=B(a,r)$.
Let $f = f_1 + f_2$, $f_1 = f \chi_{2B}$.
Then we have 
\begin{align*}
 \sup_{t\in(0,\infty)}
 \Phi(t)m(f_1,t)
&\le
 \sup_{t\in(0,\infty)}
 \Phi(t)m\left(2B,\frac{f}{\|f\|_{\Phi,\vp,2B,\weak}},t\right)
\\
&\le 
 |2B| \vp(2r)
 \le
 C_{n,\vp}|B|\vp (r),
\end{align*}
instead of \eqref{int Phi(f1)}.
Hence, \eqref{ww modular}
yields
\begin{align*}
 \sup_{t\in(0,\infty)}\Phi(t)\ m\left(B,\frac{Mf_1}{C_{n,\vp}C_{\Phi}},t\right)
 &\le
 \frac1{C_{n,\vp}}
 \sup_{t\in(0,\infty)}\Phi(t)\ m\left(\frac{Mf_1}{C_{\Phi}},t\right) \\
 &\le
 \frac1{C_{n,\vp}}
 \sup_{t\in(0,\infty)}
 \Phi(t)m(f_1,t)
 \le
 |B|\vp (r).
\end{align*}
We also have $Mf_2(x)\le C\Phi^{-1}(\vp(r))$ for $x\in B$ by Lemma~\ref{lem:Mf wLPp}.
Then
\begin{equation*}
 \sup_{t\in(0,\infty)}
 \Phi(t)\ m\left(B,\frac{Mf_{2}}{C},t\right) 
 \le
 \sup_{t\in(0,\infty)}
 \Phi(t)\ m\left(B,\Phi^{-1}(\vp(r)),t\right) 
 \le
 |B|\vp(r).
\end{equation*}
Therefore, we have the conclusion.
\end{proof}

\subsection{Proof of Theorem~\ref{thm:Ir}}\label{ss:proof thm Ir}

We may assume that $\Phi,\Psi\in\iPy$
and that $\vp$ is continuous and strictly decreasing. 
First we state a lemma.

\begin{lem}[{\cite[Proposition~1]{Deringoz-Guliyev-Nakai-Sawano-Shi2019Posi}}]\label{lem:trho}
Let $\rho,\tau:(0,\infty)\to(0,\infty)$.
Assume that $\rho$ satisfies \eqref{sup rho} 
and that $\tau$ satisfies the doubling condition \eqref{doubling}.
Define
\begin{equation}\label{eq:tilde rho}
 \trho(r)=\int_{K_1 r}^{K_2 r} \frac{\rho(t)}t\,dt,
 \quad r\in(0,\infty).
\end{equation}
Then there exists a positive constant $C$ such that, 
for all $r\in(0,\infty)$,
\begin{align}
\label{trho m}
  \sum_{j=-\infty}^{-1}\trho(2^j r) &\lesssim
 \int_{0}^{K_2 r}\frac{\rho(t)}{t}\,dt,\\
\label{trho p}
 \sum_{j=0}^{\infty}\trho(2^j r)\tau(2^{j}r)
&\lesssim \int_{K_1 r}^{\infty} \frac{\rho(t)\tau(t)}{t} \,dt.
\end{align}
\end{lem}

\begin{proof}[\bf Proof of Theorem~\ref{thm:Ir}~\ref{Ir 1}]
By the assumption \eqref{cG* dec} 
we may assume that $\vp$ is bijective from $(0,\infty)$ to itself.
If $\Phi^{-1}(0) > 0$, then,
from $\eqref{Ir A}$ it follows that 
\begin{equation*}
 0 < \int_0^{\infty} \frac{\rho(t)}{t}\,dt\;{\Phi}^{-1}(0) 
 \ls \Psi^{-1}(0), \quad
\end{equation*}
since $\dlim_{r\to\infty}\vp(r)=0$.
Let $f\in \LPp(\R^n)$ and $\|f\|_{\LPp}=1$,
and let $x\in\R^n$.
We may assume that 
\begin{equation*}
 0<\frac{Mf(x)}{C_0}<\infty 
 \quad\text{and}\quad
 0\le\Phi\left(\frac{Mf(x)}{C_0}\right)<\infty, 
\end{equation*}
otherwise there is nothing to prove. 

If $\Phi({Mf(x)}/{C_0})=0$,
then, by \eqref{inverse ineq} we have
\begin{equation*}
 \frac{Mf(x)}{C_0}
 \le \Phi^{-1}(0)
 =
 \sup\{u \ge 0: \Phi(u) = 0 \}.
\end{equation*}
Hence, using \eqref{trho m}, we have
\begin{align*}
 |\Ir f(x)|
 &\le
 \sum_{j=-\infty}^\infty
 \int_{2^{j}\le|x-y|<2^{j+1}} \frac{\rho(|x-y|)}{|x-y|^n}|f(x)|\,dx \\
 &\ls
 \sum_{j=-\infty}^\infty
 \frac{\trho(2^j)}{2^{j n}}
 \int_{|x-y|<2^{j+1}}|f(y)|\,dy
 \ls
 \int_0^\infty \frac{\rho(t)}{t}\,dt \ M f(x)\\
 &\le C_0
 \int_0^{\infty} \frac{\rho(t)}{t}\,dt \ {\Phi}^{-1}(0) 
 \ls
 \Psi^{-1}(0)
 \ls
 \Psi^{-1}\left(\Phi\left(\frac{Mf(x)}{C_0}\right)\right),
\end{align*}
which shows \eqref{Ir pointwise}.

If $\Phi({Mf(x)}/{C_0})>0$,
choose $r\in(0,\infty)$ such that
\begin{equation}\label{choose r}
 \vp(r) = \Phi\left(\frac{Mf(x)}{C_0}\right),
\end{equation}
and let
\begin{align*}
 J_1 
 &=
 \sum_{j=-\infty}^{-1}
 \frac{\trho(2^jr)}{(2^jr)^n}
 \int_{|x-y|<2^{j+1}r}|f(y)|\,dy, \\
 J_2
 &=
 \sum_{j=0}^\infty
 \frac{\trho(2^jr)}{(2^jr)^n}
 \int_{|x-y|<2^{j+1}r}|f(y)|\,dy. 
\end{align*}
Then
\begin{equation*}
 |\Ir f(x)|
 \ls
 J_1+J_2.
\end{equation*}
By \eqref{choose r} and \eqref{inverse ineq} 
we have
$Mf(x) \le C_0\Phi^{-1}(\vp(r))$.
Then, using \eqref{trho m}, we have
\begin{equation*}
 J_1
 \ls
 \int_0^{K_2r}\frac{\rho(t)}{t}\,dt \ M f(x)
 \ls
 \int_0^{K_2r}\frac{\rho(t)}{t}\,dt \ \Phi^{-1}(\vp(r)).
\end{equation*}
Next, by Lemma~\ref{lem:fint_B f}, $\|f\|_{\LPp}\le1$ and \eqref{trho p} we have
\begin{equation*}
 J_2
 \ls
 \sum_{j=0}^\infty
 \trho(2^jr)\Phi^{-1}(\vp(2^{j+1}r))
 \ls
 \int_{K_1r}^{\infty}\frac{\rho(t)\Phi^{-1}(\vp(t))}{t}\,dt.
\end{equation*}
Then
by \eqref{Ir A},
the doubling condition of $\Phi^{-1}(\vp(r))$ and $\Psi^{-1}(\vp(r))$,
and \eqref{choose r}, we have
\begin{align*}
 J_1+J_2
 &\ls
 \int_0^{K_2r}\frac{\rho(t)}{t}\,dt \ \Phi^{-1}(\vp(r))
 +\int_{K_1r}^{\infty}\frac{\rho(t)\Phi^{-1}(\vp(t))}{t}\,dt \\
 &\ls
 \frac{\Psi^{-1}(\vp(K_2r))}{\Phi^{-1}(\vp(K_2r))} \Phi^{-1}(\vp(r))
 +\Psi^{-1}(\vp(K_1r)) \\
 &\sim
 \Psi^{-1}(\vp(r))
 =
 \Psi^{-1}\left(\Phi\left(\frac{Mf(x)}{C_0}\right)\right).
\end{align*}
Combining this inequality with \eqref{inverse ineq}, we have \eqref{Ir pointwise}.

The proof of \eqref{Ir pointwise w} is almost the same as one of \eqref{Ir pointwise}.
The only difference is that
we use Lemma~\ref{lem:fint_B wLPp} instead of Lemma~\ref{lem:fint_B f} to estimate $J_2$.
\end{proof}

To prove Theorem~\ref{thm:Ir} (ii) we state three lemmas.

\begin{lem}[{\cite[Lemma~2.1]{Eridani-Gunawan-Nakai-Sawano2014MIA}}]\label{lem:Ir chi BR}
There exists a positive constant $C$ such that, 
for all $R>0$, 
\begin{equation*}
 \int_0^{R/2} \frac{\rho(t)}{t}\,dt \ \chi_{B(0,R/2)}(x)
 \le
 C\Ir \chi_{B(0,R)}(x),
 \quad x\in\R^n.
\end{equation*}
\end{lem}

The following lemma is an extension of \cite[Lemma~2.4]{Eridani-Gunawan-Nakai-Sawano2014MIA}
and gives a typical element in $\LPp(\R^n)$.

\begin{lem}\label{lem:g}
For $\Phi\in\iPy$ and $\vp\in\cGdec$, 
let $g(x)=\Phi^{-1}(\vp(|x|))$.
If $\vp$ satisfies \eqref{int vp tn-1},
then $g\in\LPp (\R^n)$.
\end{lem}

\begin{proof}
By Remark~\ref{rem:vp bijective} we may assume that 
$\vp$ is decreasing. 
In this case $x \mapsto \vp(|x|)$ is radial decreasing, 
so that, for any ball $B(a,r)$,
\begin{equation*}
 \fint_{B(a,r)} \vp(|x|)\,dx
 \le
 \fint_{B(0,r)} \vp(|x|)\,dx
 \sim
 \frac1{r^n}\int_0^r \vp(t)t^{n-1}\,dt
 \ls
 \vp(r).
\end{equation*}
In the above we used \eqref{int vp tn-1} for the last inequality.
Then, taking a suitable constant $C_g\ge1$,
and using the convexity of $\Phi$ and \eqref{inverse ineq}, we have
\begin{align*}
 \frac1{\vp(r)}\fint_{B(a,r)} \Phi \left( \frac{\Phi^{-1}(\vp(|x|))}{C_g} \right) dx
 &\le
 \frac1{C_g\vp(r)}\fint_{B(a,r)} \Phi(\Phi^{-1}(\vp(|x|))) \,dx \\
 &\le
 \frac1{C_g\vp(r)}\fint_{B(a,r)} \vp(|x|)) \,dx
 \le1.
\end{align*}
This shows $g\in\LPp(\R^n)$ with
$\|g\|_{\LPp}\le C_g$.
\end{proof}

\begin{lem}\label{lem:gR}
Let $\Phi\in\iPy$ and $\vp\in\cGdec$.
For $R>0$, let 
$$g_R(x)=\Phi^{-1}(\vp(|x|))\chi_{\R^n\setminus B(0,R)}(x).$$
Then there exists a positive constant $C$ such that, for all $R>0$,
\begin{equation*}
 \int_{2R}^{\infty}\frac{\rho(t)\Phi^{-1}(\vp(t))}{t}\,dt \ \chi_{B(0,R)}(x)
 \le 
 C\Ir g_R(x).
\end{equation*}
\end{lem}

\begin{proof}
Let $x\in B(0,R)$.
Then $B(0,R)\subset B(x,2R)$ 
and 
$|x-y| \sim |y|$ for all $y\notin B(0,2R)$.
Hence
\begin{align*}
 \Ir g_R(x)
 &=
 \int_{\R^n \setminus B(0,R)} \frac{\rho(|x-y|)\Phi^{-1}(\vp(|y|))}{|x-y|^n}\,dy \\
 &\ge
 \int_{\R^n \setminus B(x,2R)} \frac{\rho(|x-y|)\Phi^{-1}(\vp(|y|))}{|x-y|^n}\,dy \\
 &=
 \int_{\R^n \setminus B(0,2R)} \frac{\rho(|y|)\Phi^{-1}(\vp(|x-y|))}{|y|^n}\,dy \\
 &\sim
 \int_{2R}^{\infty}\frac{\rho(t)\Phi^{-1}(\vp(t))}{t}\,dt.
\end{align*}
This shows the conclusion.
\end{proof}

\begin{proof}[\bf Proof of Theorem~\ref{thm:Ir}~\ref{Ir 2}]
Firstly, by Lemma~\ref{lem:Ir chi BR} and the boundedness of $\Ir$, we have
\begin{align*}
 &\int_0^r \frac{\rho(t)}{t}\,dt \|\chi_{B(0,r)}\|_{\wLPsp}
 \ls
 \| \Ir\chi_{B(0,2r)}\|_{\wLPsp}
 \ls
 \| \chi_{B(0,2r)}\|_{\LPp}.
\end{align*}
By Lemma~\ref{lem:chi norm 2} 
and the doubling condition of $\Phi^{-1}(\vp(r))$ 
we have
\begin{equation*}
 \int_0^r\frac{\rho(t)}{t}\,dt\;{\Phi}^{-1}(\vp(r)) 
 \ls
 \Psi^{-1}(\vp(r)).
\end{equation*}
Secondly, under the assumption \eqref{int vp tn-1},
let $g$ and $g_R$ be functions as in Lemmas~\ref{lem:g} and \ref{lem:gR}, respectively.
Then 
by the boundedness of $\Ir$ we obtain
\begin{equation*}
 \int_{r}^{\infty} \frac{\rho(t)\Phi^{-1}(\vp(t))}{t}dt \|\chi_{B(0,r/2)}\|_{\wLPsp}
 \ls
 \|\Ir g_{r/2}\|_{\wLPsp}
 \ls
 \|g_{r/2}\|_{\LPp}
 \ls
 \|g\|_{\LPp}.
\end{equation*}
By Lemma~\ref{lem:chi norm 2} and the doubling condition of $\Psi^{-1}(\vp(r))$ we have
\begin{equation*}
 \int_r^{\infty} \frac{\rho(t)\Phi^{-1}(\vp(t))}{t}dt
 \ls
 \Psi^{-1}(\vp(r)).
\end{equation*}
Thus, we obtain the conclusion.
\end{proof}

\begin{proof}[\bf Proof of Theorem~\ref{thm:Ir}~\ref{Ir 3}]
By Lemma~\ref{lem:chi norm 2}, Corollary~\ref{cor:chi norm} and Remark~\ref{rem:wLPp}
we have
\begin{equation*}
 \|\chi_{B(0,r)}\|_{\wLPsp}
 \sim
 \|\chi_{B(0,r)}\|_{L^{(1,\Psi^{-1}(\vp))}}
 \sim
 \frac1{\Psi^{-1}(\vp(r))}.
\end{equation*}
Then we can replace $\|\chi_{B(0,r)}\|_{\wLPsp}$ 
by $\|\chi_{B(0,r)}\|_{L^{(1,\Psi^{-1}(\vp))}}$
in the proof of \ref{Ir 2}.
\end{proof}

\subsection{Proof of Theorem~\ref{thm:Mr}}\label{ss:proof thm Mr}

We may assume that $\Phi,\Psi\in\iPy$
and that $\vp$ is continuous and strictly decreasing. 

\begin{proof}[\bf Proof of Theorem~\ref{thm:Mr}~\ref{Mr 1}]
The pointwise estimate \eqref{Mr pointwise} 
was already proven in \cite[Theorem~5.1]{Shi-Arai-Nakai2020Banach}.
The pointwise estimate \eqref{Mr pointwise w} can be proven 
by almost the same way as \eqref{Mr pointwise}.
To prove \eqref{Mr pointwise w} we use
\begin{equation*}
 \fint_{B}|f(x)|\,dx
 \le
 C\Phi^{-1}(\vp(r)) \|f\|_{\wLPp}
\end{equation*}
by Lemma~\ref{lem:fint_B wLPp}
instead of 
\begin{equation*}
 \fint_{B}|f(x)|\,dx
 \le
 C\Phi^{-1}(\vp(r)) \|f\|_{\LPp}.
\end{equation*}
Other parts are the same as the proof of \eqref{Mr pointwise}.
\end{proof}

For the proof of Theorem~\ref{thm:Mr} {\rm (ii)}
we use the following lemma.

\begin{lem}[{\cite[Lemma~5.1]{Shi-Arai-Nakai2019Taiwan}}]\label{lem:Mr chi}
Let $\rho:(0,\infty)\to(0,\infty)$.
Then, for all $r\in(0,\infty)$,
\begin{equation}\label{sup rho < Mr}
 \left(\sup_{0<t\le r}\rho(t)\right)\chi_{B(0,r)}(x)
 \le (\Mr\chi_{B(0,r)})(x),
 \quad x\in\R^n.
\end{equation}
\end{lem}

\begin{proof}[\bf Proof of Theorem~\ref{thm:Mr}~\ref{Mr 2}]
By Lemma~\ref{lem:Mr chi} and the boundedness of $\Mr$ 
from $\LPp(\R^n)$ to $\wLPsp(\R^n)$
we have
\begin{equation*}
 \left(\sup_{0<t\le r}\rho(t)\right)\|\chi_{B(0,r)}\|_{\wLPsp}
 \le
 \|\Mr\chi_{B(0,r)}\|_{\wLPsp}
 \ls
 \|\chi_{B(0,r)}\|_{\LPp}.
\end{equation*}
Then, by Lemma~\ref{lem:chi norm 2}
we have the conclusion.
\end{proof}

\begin{proof}[\bf Proof of Theorem~\ref{thm:Mr}~\ref{Mr 3}]
The same as Proof of Theorem~\ref{thm:Ir}~\ref{Ir 3}.
\end{proof}

\section*{Acknowledgement}
The second author was supported by Grant-in-Aid for Scientific Research (B), 
No.~15H03621, Japan Society for the Promotion of Science.


\bigskip

\begin{flushright}
\begin{minipage}{100mm}
\noindent
Ryota Kawasumi \\
Minohara 1-6-3 (B-2), Misawa, Aomori 033-0033, Japan \\
rykawasumi@gmail.com
\\[3ex]
\noindent
Eiichi Nakai \\
Department of Mathematics \\
Ibaraki University \\
Mito, Ibaraki 310-8512, Japan \\
eiichi.nakai.math@vc.ibaraki.ac.jp  
\\[3ex]
\noindent
Minglei Shi \\
Jiangxi University of Engineering \\
Xinyu, Jiangxi 338000, China \\
shimingleiyy@163.com
\end{minipage}
\end{flushright}


\begin{thebibliography}{99}




\bibitem{Adams1975}
D.~R.~Adams, 
{A note on Riesz potentials}, 
Duke Math. J. 42 (1975), No.~4, 765--778. 



\bibitem{Chiarenza-Frasca1987}
F.~Chiarenza and M.~Frasca,  
{Morrey spaces and Hardy-Littlewood maximal function}, 
Rend. Mat. Appl. (7) 7 (1987), No.~3-4, 273--279. 

\bibitem{Cianchi1999}
A.~Cianchi,  
{Strong and weak type inequalities for some classical operators 
in Orlicz spaces}, 
J. London Math. Soc. (2) 60 (1999), No.~1, 187--202.


\bibitem{Deringoz-Guliyev-Nakai-Sawano-Shi2019Posi} 
F.~Deringoz, V.S.~Guliyev, E.~Nakai, Y.~Sawano and M.~Shi,
Generalized fractional maximal and integral operators 
on Orlicz and generalized Orlicz-Morrey spaces of the third kind,
Positivity 23 (2019), No.~3, 727--757.

\bibitem{Deringoz-Guliyev-Samko2014} 
F.~Deringoz, V.S.~Guliyev and S.~Samko,
Boundedness of the maximal and singular operators on generalized Orlicz-Morrey spaces. 
Operator theory, operator algebras and applications, 139--158, Oper. Theory Adv. Appl., 242, 
Birkh\"auser/Springer, Basel, 2014. 

%

\bibitem{Eridani-Gunawan-Nakai-Sawano2014MIA} 
Eridani, H.~Gunawan, E.~Nakai and Y.~Sawano,
Characterizations for the generalized fractional
integral operators on Morrey spaces,
Math. Inequal. Appl. 17 (2014), No.~2, 761--777.

\bibitem{Ferreira2016}  
L.~C.~F.~Ferreira, 
On a bilinear estimate in weak-Morrey spaces and uniqueness for Navier-Stokes equations, 
J. Math. Pures Appl. (9) 105 (2016), No.~2, 228--247. 

%

\bibitem{Gala-Sawano-Tanaka2015}  
S.~Gala, Y.~Sawano and H.~Tanaka, 
A remark on two generalized Orlicz-Morrey spaces, 
J. Approx. Theory 198 (2015), 1--9. 

\bibitem{Gallardo1988}
D.~Gallardo, 
Orlicz spaces for which the Hardy-Littlewood maximal operator is bounded,
Publ. Mat. 32 (1988), No.~2, 261--266. 


\bibitem{Gogatishvili-Mustafayev-Agcayazi2018}
A.~Gogatishvili, R.~Mustafayev and M.~Agcayazi, 
Weak-type estimates in Morrey spaces for maximal commutator and commutator of maximal function, 
Tokyo J. Math. 41 (2018), No.~1, 193--218. 


\bibitem{Guliyev-Hasanov-Sawano-Noi2016}
V.~S.~Guliyev, S.~G.~Hasanov, Y.~Sawano and T.~Noi, 
Non-smooth atomic decompositions for generalized Orlicz-Morrey spaces of the third kind,
Acta Appl. Math. 145 (2016), No.~1, 133--174. 

\bibitem{Gunawan2003} 
H.~Gunawan, 
A note on the generalized fractional integral operators, 
J. Indones. Math. Soc. 9 (2003), No.~1, 39--43.

\bibitem{Gunawan-Hakim-Limanta-Masta2017}
H.~Gunawan, D.~I.~Hakim, K.~M.~Limanta and A.~A.~Masta,
Inclusion properties of generalized Morrey spaces,
Math. Nachr. 290 (2017), No. 2-3, 332--340. 

\bibitem{Gunawan-Hakim-Nakai-Sawano2018}
H.~Gunawan, D.~I.~Hakim, E.~Nakai and Y.~Sawano.
On inclusion relation between weak Morrey spaces and Morrey spaces, 
Nonlinear Anal. 168  (2018), 27--31. 

\bibitem{Hakim-Sawano2016}
D.~I.~Hakim and Y.~Sawano,
Interpolation of generalized Morrey spaces,
Rev. Mat. Complut. 29 (2016), No.~2, 295--340.

\bibitem{Hatano2019preprint}
N.~Hatano, 
Boundedness of the Calder\'on-Zygmund operators and 
their commutators on Morrey-Lorentz spaces, 
preprint.

\bibitem{Ho2013} 
K.-P.~Ho, 
Vector-valued maximal inequalities on weighted Orlicz-Morrey spaces, 
Tokyo J. Math. 36 (2013), No.~2, 499--512.

\bibitem{Ho2016}  
K.-P.~Ho, 
Vector-valued operators with singular kernel and Triebel-Lizorkin block spaces with variable exponents, 
Kyoto J. Math. 56 (2016), No.~1, 97--124. 

\bibitem{Ho2017}
K.-P.~Ho, 
Atomic decompositions and Hardy's inequality on weak Hardy-Morrey spaces, 
Sci. China Math. 60 (2017), No.~3, 449--468. 

\bibitem{Ho2019}  
K.-P.~Ho, 
Weak type estimates of singular integral operators on Morrey-Banach spaces,
Integral Equations Operator Theory 91 (2019), No.~3, Paper No.~20. 

%

\bibitem{Jiao-Peng-Liu2008}
Y.~Jiao, L.~H.~Peng and P.~D.~Liu,
Interpolation for weak Orlicz spaces with $M_{\Delta}$ condition. 
Sci. China Ser. A 51 (2008), No.~11, 2072--2080. 

\bibitem{Kawasumi-Nakai2020Hiroshima}
R.~Kawasumi and E.~Nakai,
Pointwise multipliers on weak Orlicz spaces,
Hiroshima Math. J. 50 (2020), No.~2, 169--184.

\bibitem{Kita1996}
H.~Kita, 
On maximal functions in Orlicz spaces,
Proc. Amer. Math. Soc. 124 (1996), No.~10, 3019--3025. 

\bibitem{Kita1997}
H.~Kita,
On Hardy-Littlewood maximal functions in Orlicz spaces, 
Math. Nachr. 183 (1997), 135--155.

\bibitem{Kita2009}
H.~Kita,
Orlicz spaces and their applications (Japanese),
Iwanami Shoten, Publishers. Tokyo, 2009.

\bibitem{Kokilashvili-Krbec1991}
V.~Kokilashvili and M.~Krbec,
{Weighted inequalities in Lorentz and Orlicz spaces},
World Scientific Publishing Co., Inc., River Edge, NJ, 1991. 

%

\bibitem{Krein-Petunin-Semenov1982}
S.~G.~Kre\u{\i}n, Yu.~\={I}.~Petun\={\i}n and E.~M.~Sem\"{e}nov, 
Interpolation of linear operators,
Translated from the Russian by J.~Sz\H{u}cs,
Translations of Mathematical Monographs, 54, 
American Mathematical Society, Providence, R.I., 1982. 

\bibitem{Krasnoselsky-Rutitsky1961}
M.~A.~Krasnosel'ski\u{\i} and Ja.~B.~Ruticki\u{\i}, 
Convex functions and Orlicz spaces, 
Translated from the first Russian edition by Leo~F.~Boron, 
P. Noordhoff Ltd., Groningen 1961.

\bibitem{Li2012}
H.~Li, 
Hardy-type inequalities on strong and weak Orlicz-Lorentz spaces, 
Sci. China Math. 55 (2012), No.~12, 2493--2505. 

\bibitem{Liang-Yang-Jiang2016}
Y.~Liang, D.~Yang and R.~Jiang, 
Weak Musielak-Orlicz Hardy spaces and applications,
Math. Nachr. 289 (2016), No.~5-6, 634--677. 

\bibitem{Liu-Wang2013}
P.~Liu and M.~Wang,
Weak Orlicz spaces: Some basic properties and their applications to harmonic analysis,
Sci. China Math. 56 (2013), 789--802.

\bibitem{Maligranda1989}
L. Maligranda,
Orlicz spaces and interpolation,
Seminars in mathematics 5,
Departamento de Matem\'atica, Universidade Estadual de Campinas, Brasil, 1989.

\bibitem{Mizuhara1990}
T.~Mizuhara, 
Boundedness of some classical operators on generalized Morrey spaces,
Harmonic analysis (Sendai, 1990), 183--189, 
ICM-90 Satell. Conf. Proc., Springer, Tokyo, 1991. 



\bibitem{Mizuta-Nakai-Ohno-Shimomura2010JMSJ}
Y.~Mizuta, E.~Nakai, T.~Ohno and T.~Shimomura, 
Boundedness of fractional integral operators on Morrey spaces and Sobolev embeddings 
for generalized Riesz potentials, 
J. Math. Soc. Japan 62 (2010), No.~3, 707--744. 

\bibitem{Morrey1938}
C.~B.~Morrey,  
On the solutions of quasi-linear elliptic partial differential equations,
Trans. Amer. Math. Soc. 43 (1938), No.~1, 126--166.


\bibitem{Nakai1994MathNachr}
E. Nakai, 
Hardy-Littlewood maximal operator, singular integral operators 
and the Riesz potentials on generalized Morrey spaces, 
Math. Nachr. 166 (1994), 95--103.


\bibitem{Nakai2000ISAAC}
E.~Nakai, 
On generalized fractional integrals in the Orlicz spaces, 
Proceedings of the Second ISAAC Congress, Vol. 1 (Fukuoka, 1999), 75--81, 
Int. Soc. Anal. Appl. Comput., 7, Kluwer Acad. Publ., Dordrecht, 2000. 

\bibitem{Nakai2001Taiwan}
E. Nakai, 
On generalized fractional integrals, 
Taiwanese J. Math. 5 (2001), 587--602.

\bibitem{Nakai2001SCMJ}
E. Nakai, 
On generalized fractional integrals in the Orlicz spaces 
on spaces of homogeneous type, 
Sci. Math. Jpn. 54 (2001), No.~3, 473--487.

\bibitem{Nakai2002Lund}
E. Nakai,
On generalized fractional integrals on the weak Orlicz spaces, 
${\mathrm{BMO}}_{\phi}$, the Morrey spaces and the Campanato spaces, 
Function spaces, interpolation theory and related topics (Lund, 2000), 389--401, 
de Gruyter, Berlin, 2002.

\bibitem{Nakai2004KIT}
E. Nakai,
Generalized fractional integrals on Orlicz-Morrey spaces, 
Banach and Function Spaces, 323--333, 
Yokohama Publ., Yokohama, 2004.


\bibitem{Nakai2008Studia}
E.~Nakai,
Orlicz-Morrey spaces and the Hardy-Littlewood maximal function,
Studia Math. 188 (2008), No~3, 193--221.

\bibitem{Nakai2008KIT}
E.~Nakai, 
Calder\'on-Zygmund operators on Orlicz-Morrey spaces and modular inequalities, 
Banach and function spaces II, 393--410, Yokohama Publ., Yokohama, 2008. 


\bibitem{Nakai-Sumitomo2001SCMJ}
E.~Nakai and H.~Sumitomo,
On generalized Riesz potentials and spaces of some smooth functions,
Sci. Math. Jpn. 54 (2001), No.~3, 463--472.


\bibitem{ONeil1965}
R.~O'Neil,
{Fractional integration in Orlicz spaces. I.}, 
Trans. Amer. Math. Soc. 115 (1965), 300--328. 

\bibitem{Orlicz1932}
W.~Orlicz,
{\"Uber eine gewisse Klasse von R\"aumen vom Typus B},
Bull. Acad. Polonaise A (1932), 207--220; 
reprinted in his Collected Papers, PWN, Warszawa 1988, 217--230.

\bibitem{Orlicz1936}
W.~Orlicz,
{\"Uber R\"aume $(L^M)$},
Bull. Acad. Polonaise A (1936), 93--107; 
reprinted in his Collected Papers, PWN, Warszawa 1988, 345--359.

\bibitem{Peetre1966}
J. Peetre, 
On convolution operators leaving $L^{p,\lambda}$ spaces invariant, 
Ann. Mat. Pura Appl. (4) 72 (1966), 295--304. 

\bibitem{Peetre1969}
J. Peetre, 
On the theory of $\mathcal L_{p,\lambda}$ spaces,
J. Funct. Anal. 4 (1969), 71--87. 

\bibitem{Perez1994}
C. P\'erez,
Two weighted inequalities for potential and fractional type maximal operators, 
Indiana Univ. Math. J. 43 (1994), No.~2, 663--683.

\bibitem{Rao-Ren1991}
M.~M.~Rao and Z.~D.~Ren,
Theory of Orlicz Spaces,
Marcel Dekker, Inc., New York, Basel and Hong Kong, 1991.

\bibitem{Sawano2019} 
Y.~Sawano, 
Singular integral operators acting on Orlicz-Morrey spaces of the first kind,
Nonlinear Stud.  26  (2019),  No.~4, 895--910. 

\bibitem{Sawano-ElShabrawy2018}
Y.~Sawano and S.~R.~El-Shabrawy, 
Weak Morrey spaces with applications, 
Math. Nachr. 291 (2018), No.~1, 178--186.

\bibitem{Sawano-Sugano-Tanaka2011} 
Y.~Sawano, S.~Sugano and H.~Tanaka, 
Generalized fractional integral operators and fractional maximal operators 
in the framework of Morrey spaces, 
Trans. Amer. Math. Soc. 363 (2011), No.~12, 6481--6503.

\bibitem{Sawano-Sugano-Tanaka2012}
Y.~Sawano, S.~Sugano and H.~Tanaka,  
Orlicz-Morrey spaces and fractional operators,
Potential Anal. 36 (2012), No.~4, 517--556. 

\bibitem{Shi-Arai-Nakai2019Taiwan}
M.~Shi, R.~Arai and E.~Nakai, 
Generalized fractional integral operators and their commutators with functions 
in generalized Campanato spaces on Orlicz spaces, 
Taiwanese J. Math. 23 (2019), No.~6, 1339--1364.



\bibitem{Shi-Arai-Nakai2020Banach}
M.~Shi, R.~Arai and E.~Nakai,
Commutators of integral operators with functions in Campanato spaces on Orlicz-Morrey spaces, Banach J. Math. Anal. 15 (2021), Paper No. 22, 41 pp.




%

\bibitem{Sihwaningrum-Sawano2013}
I.~Sihwaningrum and Y.~Sawano,  
Weak and strong type estimates for fractional integral operators on Morrey spaces over metric measure spaces,
Eurasian Math. J. 4 (2013), No.~1, 76--81. 

\bibitem{Stein1970}
E.~M.~Stein, 
Singular integrals and differentiability properties of functions, 
Princeton Mathematical Series, No.~30, 
Princeton University Press, Princeton, N.J. 1970. 


\bibitem{Torchinsky1976}
A.~Torchinsky,  
Interpolation of operations and Orlicz classes, 
Studia Math. 59 (1976), No.~2, 177--207. 

\bibitem{Torchinsky1986}
A.~Torchinsky,  
Real-variable methods in harmonic analysis, Academic Press, New York, 1986.

\bibitem{Tsereteli1969}
O.~D.~Tsereteli (Cereteli), 
The interpolation of operators by the truncation method, (Russian) 
Sakharth. SSR Mecn. Akad. Math. Inst. \v{S}rom. 36 (1969), 111--122. 

%

\bibitem{Tumalun-Gunawan2019}
N.~K.~Tumalun and H.~Gunawan, 
Morrey spaces are embedded between weak Morrey spaces and Stummel classes, 
J. Indones. Math. Soc. 25 (2019), No.~3, 203--209.


\bibitem{Xie-Yang2019}
G.~Xie and D.~Yang, 
Atomic characterizations of weak martingale Musielak-Orlicz Hardy spaces 
and their applications,
Banach J. Math. Anal. 13 (2019), No.~4, 884--917. 


\end{thebibliography}
\end{document}